\documentclass[12pt]{amsart}
\usepackage{graphicx}
\usepackage{amssymb}
\usepackage{amsfonts}
\usepackage{hyperref}
\usepackage{cancel}

\usepackage[normalem]{ulem}

\usepackage[margin=3cm]{geometry}
\usepackage{mathrsfs}
\swapnumbers
\usepackage[usenames,dvipsnames,svgnames,table]{xcolor}

  [1995/08/01 v0.1 Double stroke roman fonts]

\def\ds@whichfont{dsrom}
\relax

\DeclareMathAlphabet{\mathds}{U}{\ds@whichfont}{m}{n}

\numberwithin{equation}{section}
\newtheorem{theorem}{Theorem}%[section]
\newtheorem{lemma}[theorem]{Lemma}
\newtheorem{corollary}[theorem]{Corollary}

\newtheorem{proposition}[theorem]{Proposition}
\theoremstyle{definition}
\newtheorem{definition}[theorem]{Definition}
\newtheorem{assumption}[theorem]{Assumption}
\newtheorem{remark}[theorem]{Remark}

\theoremstyle{plain}

\numberwithin{figure}{section} %% Comment out for sequentially-numbered
\theoremstyle{plain}
\theoremstyle{plain}
\theoremstyle{remark}
\newtheorem*{acknowledgement*}{Acknowledgement}
\theoremstyle{example}

\newcommand{\cA}{{\mathcal A}}

\newcommand{\cE}{{\mathcal E}}
\newcommand{\cF}{{\mathcal F}}

\newcommand{\cP}{{\mathcal P}}

\newcommand{\cR}{{\mathcal R}}
\newcommand{\cS}{{\mathcal S}}

\newcommand{\cX}{{\mathcal X}}

\newcommand{\te}{{\theta}}

\newcommand{\Om}{{\Omega}}
\newcommand{\om}{{\omega}}
\newcommand{\ve}{{\varepsilon}}
\newcommand{\del}{{\delta}}

\newcommand{\Gam}{{\Gamma}}

\newcommand{\sig}{{\sigma}}
\newcommand{\al}{{\alpha}}
\newcommand{\be}{{\beta}}
\newcommand{\ka}{{\kappa}}
\newcommand{\la}{{\lambda}}

\newcommand{\bbC}{{\mathbb C}}
\newcommand{\bbE}{{\mathbb E}}

\newcommand{\bbN}{{\mathbb N}}
\newcommand{\bbP}{{\mathbb P}}
\newcommand{\bbR}{{\mathbb R}}

\newcommand{\bbZ}{{\mathbb Z}}

\def\fl{\mathfrak{l}}

\def\fv{\mathfrak{v}}

\def\hc{{\hat c}}

\newcommand{\Area}{\mathrm{Area}}

\def\fc{{\mathfrak{c}}}
\def\fh{{\mathfrak{h}}}
\def\fm{{\mathfrak{m}}}
\def\fK{{\mathfrak{K}}}

\def\tf{{\tilde f}}

\def\fd{{\mathfrak{d}}}
\def\fP{{\mathfrak{P}}}
\def\fu{{\mathfrak{u}}}

\newcommand{\DS}{\displaystyle}
\newcommand{\eps}{{\varepsilon}}

\begin{document}
\title[]{A Berry-Esseen theorem and Edgeworth expansions for uniformly elliptic inhomogeneous Markov chains}
 %\vskip 0.1cm
 \author{Dmitry Dolgopyat and Yeor Hafouta}
\address{University of Maryland and The Ohio State University}

%\email{yeor.hafouta@mail.huji.ac.il}%

\thanks{DD was partially supported by the NSF and the BSF}
%\subjclass[2010]{37C30; 37C40; 37H99; 60F17}%
%\keywords{limit theorems; Almost sure central limit theorem Perron-Frobenius theorem; thermodynamic formalism; sequential dynamical systems; time dependent dynamical systems; random dynamics; }%%
\dedicatory{  }
 \date{\today}
\maketitle

\begin{abstract}\noindent
We prove a Berry-Esseen theorem and Edgeworth expansions for partial sums of the form 
$\DS S_N=\sum_{n=1}^{N}f_n(X_n,X_{n+1})$, where $\{X_n\}$ is a uniformly elliptic inhomogeneous Markov chain and $\{f_n\}$ is a sequence of uniformly bounded functions. The Berry-Esseen theorem 
holds without additional assumptions, while expansions of order $1$ hold when $\{f_n\}$ is irreducible, which is an optimal condition. 
 For higher order expansions, we then focus on two situations. The first is when the essential supremum of $f_n$ is of order $O(n^{-\be})$ for some $\be\in(0,1/2)$. In this case it turns out that expansions of any order $r<\frac1{1-2\be}$ hold, and this condition is optimal. 
The second case is uniformly elliptic chains
on a compact Riemannian manifold. When $f_n$ are uniformly Lipschitz continuous we show that $S_N$ admits expansions of all orders. When $f_n$ are uniformly H\"older continuous with some exponent $\al\in(0,1)$, we show that $S_N$ admits expansions of all orders $r<\frac{1+\al}{1-\al}$.  For H\"older continues functions with $\al<1$ our results are new also for uniformly elliptic homogeneous Markov chains and a single functional $f=f_n$. In fact, we show that the condition $r<\frac{1+\al}{1-\al}$ is optimal 
 even in the  homogeneous  case.
 \end{abstract}

\section{Introduction}
Let $Y_1,Y_2,Y_3,...$ be
 a uniformly bounded sequences of independent random variables.
Set $\DS \bar S_N=\sum_{n=1}^N(Y_n-\bbE(Y_n))$, $V_N=\text{Var}(S_N)$ and $\sig_N=\sqrt{V_N}$. The classical central limit theorem (CLT) states that if $\sig_N\to\infty$ then, as $N\to\infty$, the distribution of  $\hat S_N=\bar S_N/\sig_N$ converges to the standard normal distribution. A related classical result is the Berry-Esseen theorem \cite{Ess} which is a quantification of the CLT stating that there is an absolute constant $C_0>0$ so that for every $N\geq1$ 
\begin{equation}
\label{IndBerryEsseen0}
\sup_{t\in\bbR}\left|\bbP(\hat S_N\leq t)-\Phi(t)\right|\leq C_0\sig_N^{-3}\sum_{j=1}^N\bbE\big[\big|Y_j-\bbE[Y_j]\big|^3\big]
\end{equation}
where $\Phi$ is the standard normal distribution function  (we refer to \cite{Berry} for similar result obtained simultaneously). In \cite{Esseen1956}, Esseen proved, in particular, that 
 the optimal constant $C_0$ in the RHS of \eqref{IndBerryEsseen0}
is greater  than $0.4$. Since then there were many efforts to provide close to tight upper bounds on $C_0$, and currently the smallest possible known choice for $C_0$ is $C_0=0.56$, see \cite{Shev} and references therein. In particular, when $Y_n$ are uniformly bounded then with $\|Y\|_\infty=\sup_n\|Y_n\|_\infty$ we have
\begin{equation}
\label{IndBerryEsseen}
\sup_{t\in\bbR}\left|\bbP(\hat S_N\leq t)-\Phi(t)\right|\leq C_0 \|Y\|_\infty\sig_N^{-1}.
\end{equation} 

It turns out that the rate of $\sig_N^{-1}$ in \eqref{IndBerryEsseen} is optimal, see below.
By now the optimal convergence rate in the CLT was obtained for wide classes of stationary Markov chains \cite{Nag1, Nag2, HH} and 
other weakly dependent random processes including 
chaotic dynamical systems \cite{RE,GH,HH,GO, Jir0,Jir}, 
uniformly bounded stationary sufficiently fast $\phi$-mixing sequences \cite{Rio},
$U$-statistics \cite{CJ,GS} and locally dependent random variables \cite{BR,BC,CS} 
(the last three papers use Stein's method).

The rate $\sig_N^{-1}$ is optimal for two reasons.
First,  for  the lattice random variables the distribution function
$t\mapsto \bbP(\hat S_N\leq t)$ has jumps of order $\sig_N^{-1}$.
Secondly even if the distributions of the summands have smooth densities the rate of convergence is
still $O\left(\sig_N^{-1}\right)$ if the third moment of the sum is different from Gaussian. To address the 
moment obstacle one could introduce appropriate corrections\footnote{In the case the arithmeticity
obstacle is present, that is, the distribution is lattice, 
one can consider asymptotic expansions of $\bbP(S_N=k)$ see \cite{Feller, GK, IL, DH} and references
wherein.}.
Namely, fix $r\geq1$. We say that the {\em Edgeworth expansions of order $r$} hold if there are polynomials $P_{1,N},...,P_{r,N}$ with degrees not depending on $N$ and coefficients uniformly bounded in $N$ so that 
\begin{equation}
\label{NonStEE}
\sup_{t\in\bbR}\big|\bbP(\hat S_N\leq t)-\Phi(t)-\sum_{j=1}^r \sig_N^{-j}P_{j,N}(t)\phi(t)\big|=o(\sig_N^{-r})
\end{equation}
where $\phi(t)=\frac{1}{\sqrt{2\pi}}e^{-t^2/2}$ is the standard normal density function. 
These expansions provide a more accurate approximations of the distribution function of $\hat S_N$  in comparison with the Berry-Esseen theorem. 

For independent random variables it was proven by Esseen in \cite{Ess}, 
that the expansion of order 1 holds iff
the distribution of $S_N$ is non-lattice. The conditions for higher order expansions are not yet completely understood.
Sufficient conditions for the Edgeworth expansions of an arbitrary order
 were first obtained in \cite{Cr28} under the assumption that the characteristic
function 
of the sum $\bbE(e^{itS_N})$ 
decays exponentially in $N$ uniformly for large $t$.
Later the same expansions were obtained in \cite{Ess, Feller, BR76, Br, AP}
under 
weaker decay conditions\footnote{The decay conditions used in the above papers are optimal, since one can provide examples where the decay is slightly weaker
and there are oscillatory corrections to Edgeworth expansion, see  \cite{DF, DH}.},
where the second  paper
considered non identically distributed variables and  the fourth and fifth  considered random iid vectors. 
Later Edgeworth expansions were proven for several classes
of weakly dependent random variables including 
stationary Markov chains (\cite{Nag1, Nag2, FL}), chaotic dynamical systems 
(\cite{CP, FL, FP}) and certain classes of local statistics
(\cite{BGVZ1,Hall, BGVZ2, CJV}).
In particular,  Herv\'e-P\`ene proved in
\cite{HP}  that for several classes of stationary processes the first order 
Edgeworth expansion holds if the system is irreducible, in the sense that $S_N$ can not be
represented as $S_N'+H_N$ where $S_N'$ is lattice valued and $H_N$ is bounded.
We also mention that
in \cite{Bar, RinRot}  so called weak expansions, i.e.
expansions of the form $\bbE\left(\phi\left(S_N/\sig_N\right)\right)$ where $\phi$ is a smooth
test function were studied.
%Add more references, see the intro of \cite{AP} for instance...

Both Berry--Esseen Theorem and Edgeworth expansions require a detailed control of 
the characteristic function. 
For dependent variables, the most powerful method for analyzing the characteristic function
is the spectral approach developed by Nagaev \cite{Nag1, Nag2} (see \cite{Gou, HH} for
the detailed exposition of the spectral method). Since the spectral method relies on perturbation
theory for the spectrum of linear operators, extending it to a non stationary setting turned out
to be a non trivial task. Recently a significant progress on this problem was achieved by using
a contraction properties of the projective metric which allows to prove spectral gap type
estimates for the non-stationary compositions of linear operators (\cite{Li95, Rug, Du09, Du11}).
In particular, complex sequential Ruelle-Perron-Frobenius Theorem, proven in \cite{HK} 
provides a powerful 
tool for proving the Central Limit Theorem and its extensions in the non stationary case.
This theorem
allows to obtain both Berry--Esseen theorem (\cite{HK, HaNonl}) and Edgeworth 
expansions (\cite{Ha, DDH}) in the non stationary setting for both Markov chains and dynamical
systems.

However, the results of \cite{HK, HaNonl,Ha, DDH} are in a certain sense perturbative.
Namely, those papers study either a small perturbation of a fixed stationary system,
or they deal with random systems assuming that a system comes to a small neighborhood
of a fixed system with a positive frequency. One difficulty in studying the non-stationary
case is that there could be large cancellations of the consecutive  terms, so that
the variance of the sum, can be much smaller then the sum of the variances of the summands.
Recently \cite{DS} developed a structure theory for Markov chains which allows 
to find, for each additive functional,
a representative in the same homology class (the homologous functionals
satisfy the same limit theorems) with the smallest $L^2$ distance from either zero or from a given
lattice in~$\mathbb{R}.$  This structure theory was used in \cite{DS} to prove the
local limit theorem for non-stationary Markov chains in both diffusive and large deviations
regimes.\smallskip

In the present paper we combine the methods of \cite{HK} and \cite{DS} to obtain several
optimal results concerning the convergence rate in the CLT 
for bounded additive functional of uniformly elliptic 
non-stationary Markov chains.
Our results include
\begin{itemize}
\item Berry--Esseen bound, which holds without any additional assumptions;  
\item first order Edgeworth expansion in the irreducible case, extending theorems of Esseen and
of Herv\'e-P\`ene; 
\item higher order expansions for the chains with either decaying $L^\infty$ norm or with
bounded H\"older norm.
\end{itemize}

We emphasize that our assumptions concern only regularity of the observables.
No additional assumptions dealing with either the growth of variance or with
the decay of characteristic function away from zero are made.

The structure of the paper is the following.
Section \ref{Main} contains the precise statements of our results.
The necessary background from \cite{HK, DS} 
is given 
in Section \ref{ScBack}. 
In Section~\ref{Sec4} we discuss the Edgeworth expansions.
In general, those expansions follow from the asymptotics of the characteristic  function around $0$, together with decay of the characteristic  functions over appropriate 
domains.
In Section \ref{Sec4} we will show that  the desired expansions around the origin hold under certain  logarithmic growth conditions. We demonstrate  that under the above growth conditions the
asymptotics of the characteristic function near zero always comes from
the Edgeworth polynomials 
(regardless of whether the Edgeworth expansions hold or not). 
Those polynomials are defined canonically, and we show that under our logarithmic growth conditions the
polynomials have bounded coefficients. 
The main step in our proofs is a verification of the latter growth conditions  for the uniformly elliptic Markov chains considered in this paper. This is accomplished
in Section \ref{ScMCCum}.
Using the sequential complex Perron-Frobenius Theorem from \cite{HK}, the required estimates are obtained 
by  studying the behavior around the origin of a resulting  sequential
complex pressure functions. For independent variables the $n$-th pressure function coincides with the logarithm of the characteristic function of the $n$-th summand, and our arguments essentially reduce to the ones in \cite{Ess, Feller}.  
In comparison with \cite{HK}, where the Markov chains in random environment 
were studied, the main difficulty  is that the variance does not grow linearly fast in 
the number of summands $N$.
The Berry--Essen theorem is a direct consequence 
of the detailed asymptotics of the characteristic function near zero established in 
Section \ref{ScMCCum}. The first order expansion also follows by combining the same estimates with
the results of \cite{DS}.

In order to achieve the desired rate of decay away from $0$, an additional structure is needed. 
Thus we consider two special classes of additive functionals.
The first is when the essential supremum of the $n$-th summand converges to $0$ as $n\to\infty$. 
We show in Section \ref{ScLInf} that if 
$\DS  \|f_{n}\|_\infty=O\left(n^{-\be}\right)$ 
for some $\be\in(0,1/2)$ then the partial sums admit expansions of any order $r<\frac1{1-2\be}$, and that this condition is optimal. 
The second type of  additive functionals we consider are 
H\"older continuous functions.  If $\{X_n\}$ is a Markov chain evolving on a compact Riemannian manifold
with uniformly bounded and bounded away from $0$ densities and 
$\DS S_N=\sum_{n=1}^Nf_n(X_n,X_{n+1})$, then we show in Section \ref{ScRM}
that when $f_n$'s are uniformly bounded Lipschitz functions then $S_N$ admits Edgeworth expansions of all orders, while when $f_n$'s are uniformly bounded H\"older continuous functions with exponent $\al\in(0,1)$, then $S_N$ admits expansions of every order $r<\frac{1+\al}{1-\al}$, and that the latter condition is optimal. 
In fact, we will show that the condition $r>\frac{1+\al}{1-\al}$ is optimal even in the stationary case when $\{X_n\}$ is homogeneous  Markov chain and $f_n=f$ does not depend on $n$.

\section{Main results}\label{Main}
\subsection{A Berry-Esseen theorem and expansions of order 1}
Let $(\cX_i,\cF_i),\,i\geq1$ be a sequence of measurable spaces.
For each $i$,  let $R_i(x,dy),\,x\in\cX_i$ be a measurable family of (transition) probability measures on $\cX_{i+1}$. Let $\mu_1$ be any probability measure on $\cX_1$, and let $X_1$ be an $\cX_1$-valued random variable with distribution $\mu_1$. Let $\{X_j\}$ be the Markov 
 started from
 $X_1$ with the transition probabilities
\[
\bbP(X_{j+1}\in A|X_{j}=x)=R_{j}(x,A),
\] 
where $x\in\cX_j$ and $A\subset\cX_{j+1}$ is a measurable set.
Each $R_j$ also gives rise to a transition operator given by 
\[
R_j g(x)=\bbE[g(X_{j+1})|X_j=x]=\int g(y)R_j(x,dy)
\] 
which maps an integrable function $g$ on $\cX_{j+1}$ to an integrbale function on $\cX_j$ (the integrability is with respect to the laws of $X_{j+1}$ and $X_j$, respectively). We assume here that 
there are probability measures $\fm_j$,
%\footnote{$\mu_j$ was used to denote the law of $X_j$}, 
$j>1$ on $\cX_j$ and families of transition probabilities $p_j(x,y)$ 
so that 
\[
R_j g(x)=\int g(y)p_j(x,y)d\fm_{j+1}(y).
\]
Moreover, there exists $\ve_0>0$ so that for any $j$ we have 
\begin{equation}
\label{DUpper}
 \sup_{x,y}p_j(x,y)\leq 1/\ve_0,
 \end{equation}
  and the transition probabilities of the second step\footnote{The assumptions that 
  we have uniform lower bound on the two step density and that the summands
    $f_n$ introduced below depend
  only on two variables are taken form \cite{DS}. In fact, the arguments of \cite{DS} also work 
  in the case we have uniform ellipticity after an arbitrary fixed number of steps and $f_n$ depend
  on finitely many variables around $x_n$ require only minor modifications (but lead to a significant
  complication of the notation). On the other hand there are some new effects in the case $f$
  depends on two variables which could not be seen in the case (considered in \cite{Dob}) 
  where $f_n$ depend on a single variable. In this paper we keep the convention from \cite{DS}
  and assume two step ellipticity and two step dependence for additive functionals.
  }
   transition operators $R_j\circ R_{j+1}$ of $X_{j+2}$ given $X_j$ are bounded from below by $\ve_0$ (this is the uniform ellipticity condition): 
\begin{equation}
\label{DLower}
\inf_{j\geq1}\inf_{x,z}\int p_j(x,y)p_{j+1}(y,z)d\fm_{j+1}(y)\geq \ve_0.
\end{equation}
Next, for a uniformly bounded sequence of measurable functions $f_n:\cX_n\times\cX_{n+1}\to\bbR$ 
we set $Y_n=f_n(X_n,X_{n+1})$ and
\begin{equation}
\label{AddFunct}
S_N=\sum_{n=1}^N(Y_n-\bbE(Y_n)).
\end{equation}
Set $V_N=\text{Var}(S_N)$ and $\sig_N=\sqrt{V_N}$. Then by \cite[Theorem 2.2]{DS} we have 
$\DS \lim_{N\to\infty}V_N=\infty$ if and only if one can not decompose $Y_n$ as
$$ Y_n=\bbE(Y_n)+a_{n+1}(X_{n+1},X_{n+2})-a_{n}(X_n, X_{n+1})+g_n(X_n, X_{n+1})$$ 
 where  $a_n$ are uniformly bounded functions and $\DS \sum_n g_n(X_n, X_{n+1})$ converges almost surely.

The CLT in the case $V_N\to\infty$ is due to \cite{Dob}, see \cite{SV} for a modern proof.
Our first result here is a version of the Berry-Esseen theorem.
Denote
\begin{equation}
\label{AddFunctN}
\hat S_N=\left(S_N-\bbE[S_N]\right)/\sig_N.
\end{equation}
\begin{theorem}\label{BE}
Suppose that $\DS \lim_{N\to\infty}V_N=\infty$. Then
there is a constant $C>0$ which depends only on $\DS \sup_{n}\|Y_n\|_{L^\infty}$ and $\ve_0$ 
so that for any $N\geq1$,
\begin{equation}\label{BE INEQ}
\sup_{t\in\bbR}\left|\bbP(\hat S_N\leq t)-\Phi(t)\right|\leq C\sig_N^{-1}
\end{equation}
where $\Phi$ is the standard normal distribution function.
\end{theorem}

Next we introduce some terminology from \cite{DS}.
We say that a sequence $Z_N$ of random variables is {\em center tight} if 
there are constants $c_N$ such that $\{Z_N-c_N\}$ is tight.
Two additive functionals $f_n$ and $\tf_n$ are {\em homologous} if
$\DS \sum_{n=1}^N (f_n(X_n,X_{n+1})-\tf_n(X_n,X_{n+1}))$ is center tight.
We say that $\{f_n\}$ is {\em reducible} if it is homologous to an additive functional taking values 
in $h\mathbb{Z}$  for some $h>0$. 
If $\{f_n\}$ is not reducible, it is called {\em irreducible.}

\begin{theorem}\label{Edg1}
If $V_N$ diverges and
$\{f_n\}$ is irreducible then %for each $k\in\mathbb{Z}_+$,
$S_N$ satisfies the Edgeworth expansion of order 1, where
$$ P_{1, N}(t)=\frac{\bbE[(S_n-\bbE[S_n])^3]}{6 V_N} (t^3-3t). $$
% $P_{1,j}(t)$ is the polynomial satisfying that $P_{1,j}(t)\phi(t)$ has Fourier transform $\frac{-i\bbE[S_n^3]}{6\sig_n^2}t^3e^{-t^2/2}$.
\end{theorem}

Next, we say that $f_n$ {\em stably\footnote{
The notion of stable Edgeworth expansion 
is motivated by the notion of stable local limit theorem studied in \cite{Prok, Rozanov}. We note 
that \cite{DH} obtains conditions for the stability of Edgeworth expansions for the sums
of independent integer valued random variables (in the integer case one studies the expansions
for $\bbP(S_N=k_N)$).} 
 obeys Edgeworth expansion of order $r$} 
 if any additive functional homologous to $f_n$ satisfies Edgeworth expansions of order $r.$

\begin{corollary}
\label{CrSTEdge-Irred}
$f_n$  stably obeys Edgeworth expansion of order 1 iff it is irreducible.
\end{corollary}

\begin{proof}
If $f_n$ is irreducible then any homologous additive functional $\tf_n$ is also irreducible, so 
by Theorem \ref{Edg1}, $\tf_n$ obeys Edgeworth expansion of order 1.

If $f_n$ is reducible then its homology class contains an $h\bbZ$ valued functional $\tf_n,$ 
for some $h>0$.
By the LLT of \cite[Section 5]{DS},  $\tilde S_N$ has jumps of order $1/\sqrt{V_N}$, so $\tilde S_N$ does not 
obey expansion of order~1.
\end{proof}

\subsection{High order expansions}
\subsubsection{Summands with small essential supremum}
We  obtain the following extension of the Edgeworth expansions for function $f_n$ which converge to $0$ as $n\to\infty$.
\begin{theorem}\label{ThmEssSup}
Suppose that $\DS \lim_{N\to\infty}V_N=\infty$,  and that there are $C>0$ and $\beta\in(0,1/2)$ so that for all $n\in\bbN$ we have
$\DS \|f_n\|_\infty\leq \frac{C}{n^\beta}$. Let $r\geq1$ be an integer satisfying
\begin{equation}
\label{BetaNearHalf}
r<\frac{1}{1-2\be}.
\end{equation}
Then $S_N$ admits an Edgeworth expansion of order $r.$ In particular, if $\|f_n\|_\infty=O(n^{-1/2})$ then $S_N$ admits Edgeworth expansions of all orders.
\end{theorem}

The following result shows that the conditions of Theorem \ref{ThmEssSup} are optimal.
\begin{theorem}\label{EgThm1}
For every $\beta\in(0,\frac12)$ there exists a sequence of centered independent random variables $X_n$ so that $C_1n^{-\beta}\leq \|X_n\|_{L^\infty}\leq C_2n^{-\be}$ for some $C_1,C_2>0$ and all $n$ large enough, $V(S_N)$ is of order $N^{1-2\beta}$ but $\DS S_N=\sum_{n=1}^{N}X_n$ fails to satisfy Edgeworth expansions of any order $s$ such that $s>\frac1{1-2\beta}$.
\end{theorem}

Taking $\beta\in(0,1/4)$ we have $\frac1{1-2\beta}<2$, and we get from Theorem \ref{EgThm1} that 
$S_N$ might not admit Edgeworth expansions of order larger than $1$ if $\|f_n\|_\infty\asymp n^{-\be}$. 
%In the following section we will present stronger results when the chain evolves on a  compact manifold and $f_n$'s are sufficiently regular.

\subsubsection{Markov chains on compact Riemannian manifolds}
Let us assume that  $\{X_n\}$ is a Markov chain on a compact Riemannian manifold $M$ with transition densities $p_n(x,y)$ bounded and bounded away from $0$, uniformly in $n$.
 Let $\al\in(0,1]$ and let  $f_n: M\times M\to \bbR$ be observables satisfying $\|f_n\|_{\al}:=\max(\sup|f_n|,v_\al(f_n))\leq 1$, where $v_\al(f_n)$ is the H\"older constant of $f_n$ corresponding to the exponent $\al$. Consider the sum
$$ S_N=\sum_{n=1}^Nf_n(X_n,X_{n+1}).$$

\begin{theorem}\label{EdgThmHold}
Suppose that  $V_N=V(S_N)\to \infty.$ 
\vskip0.2cm
(i)If $\al=1$ then $S_N$ satisfies the Edgeworth expansion of all orders. 
\vskip0.2cm
(ii) If $\al<1$ then $S_N$ satisfies the Edgeworth expansion of any order $r<\frac{1+\alpha}{1-\al}$.
\end{theorem}

For smooth functions, expansions of all orders were obtained in \cite{FL} for stationary Markov chains and functions $f_n=f$ which do not depend on $n$. 
Here we have to overcome the difficulty that the variance of $f_n(X_n,X_{n+1})$ might be small, and hence the proof differs from the one in \cite{FL} even for smooth functions,  so it is also new in the stationary case.
The proof of Theorem \ref{EdgThmHold} follows the approach of \cite{DoAn}.
We note that  similar estimates are used in \cite{DoAn, DoPrev} 
to prove polynomial bounds 
for the decay of correlations for hyperbolic suspension flows with H\"older roof functions.
However, the bound of \cite{DoAn, DoPrev} are not explicit whereas here we get an explicit 
(and optimal,  see below) control on the possible location of resonances.

 We see that as $\al\to 1$, the largest order of the expansions ensured by 
 Theorem~\ref{EdgThmHold}(ii) diverges to $\infty$.
The following theorem shows that the conditions of Theorem~\ref{EdgThmHold}(ii) are optimal.
\begin{theorem}\label{EgThm2}
 Let $\{x_n\}$ be iid random variables uniformly distributed on $[-1,1]$.
For every $0<\al<1$ there exists an increasing odd function $f:[-1,1]\to[-1,1]$  which is H\"older continuous with exponent $\al$ and is onto $[-1,1]$,  so that 
$\DS S_n=\sum_{j=1}^nf(x_j)$ does not admit Edgeworth expansion of any order 
$r>\frac{1+\al}{1-\al}$.
\end{theorem}
Theorem \ref{EgThm2} show that  the conditions of Theorem~\ref{EdgThmHold}(ii) are optimal even in the stationary case. The idea in the proof of Theorem \ref{EgThm2} is to first approximate $\al$ by numbers of the form $\al_{q,p}=\ln(p)/\ln (p+q)$, for some $p,q\geq2$ so that $q|(p-1)$. Then,  the restriction of the function $f$ to $[0,1]$ will be the, so called,
Cantor function (see \cite{Gil}) corresponding to a certain Cantor set with Hausdorff dimension 
$\al_{q,p}$.
%{\color{red}
%\begin{remark}\label{Computing}
%In all the expansions described in the latter theorems, the coefficient of the polynomials $P_{j,n}$ %appearing in the  expansions of order $r$ can be computed, up to $o(\sig_n^{j-r})$ terms.
%\end{remark}
%}
\subsection{The canonical form of the Edgeworth polynomials}

We note that in the non-stationary setting, \eqref{NonStEE} does not define 
the Edgeworth polynomials uniquely
since we could always modify the coefficients by terms of order $o(\sig_N^{-r})$.
However, it turns out that one could make a canonical choice which 
a simple computation of its coefficient in a quite general setting including additive functionals 
of uniformly elliptic Markov chains considered here.

Given a nonconstant random variable $S$ with finite moments of all orders, let $a_j(S)$ denote the normalized cumulant
$$ a_j(S)=\frac{1}{V(S)i^j} \frac{d^j}{dt^j}\Big|_{t=0} \ln \left[\bbE\left(e^{it (S-\bbE(S))}\right)\right]. $$
\begin{theorem}
There exist polynomials $\fP_j(z; a_3, a_4, \dots , a_{3j})$ such that for each  
integer $r\geq1$ there is a positive constant
$\delta_r=\delta_r(\eps_0, K)$, $\DS K=\sup_n\|f_n(X_n,X_{n+1})\|_{L^\infty}$, such that if $S_N$ and $\hat S_N$ are given 
by \eqref{AddFunct} and \eqref{AddFunctN}, respectively,
then denoting 
\begin{equation}
\label{EdgeDef}
P_{j, N}(z)=\fP_j(z, a_3(S_N), \dots, a_{3j}(S_N)),
\end{equation}
$\DS \cE_{r,N}(z)=\Phi(z)+\phi(z) \sum_{j=1}^r{\sig_N^{-j}}P_{j, N}(z)$ 
and letting $\widehat\cE_{r,N}$ denote the Fourier transform of $\cE_{r,N}(z)$ we have 
\begin{equation}
\label{CharFun0}
\int_{-\delta_r\sig_N }^{\delta_r\sig_N} 
\left| \frac{\bbE\left(e^{it \hat S_N}\right)-\widehat{\cE_{r,N}}(t)}{t}\right| dt=
O\left(\sig_N^{-(1+r)}\right). 
\end{equation}
\end{theorem}
We note that our  proofs of 
Theorems \ref{Edg1}, \ref{ThmEssSup} and \ref{EdgThmHold} 
 provide the Edgeworth expansions  with the above polynomials $P_{j,n}$.

The polynomials  $\fP_j$ are given in Definition \ref{EdgPdef}. 
In \S \ref{SecVer} we show that 
for additive functionals of the Markov chains considered in this paper
the Edgeworth polynomials have bounded coefficients. 
This is done by verifying  Assumption \ref{GrowAssum} which ensures the boundness
for an abstract sequence of random variables.

We note that  \eqref{CharFun0} 
holds without any additional assumptions. However, to ensure that
the term $\cE_{r,N}(z)$ provides a good approximation to $\bbP(\hat S_N\leq z)$ we need to 
control the LHS of \eqref{CharFun0} on longer intervals of size $B\sig_N^r$ for an arbitrary $B.$
In the case $r=1$ the contribution of 
$[-B \sig_N, B \sig_N]\setminus [-\delta_1 \sig_N, \del_1 \sig_N] $
is analyzed in \cite{DS}. The case $r>1$ is addressed in Sections \ref{ScLInf} and \ref{ScRM} where we control
the characteristic function of $\hat S_N$ under the assumptions of 
Theorems \ref{ThmEssSup} and \ref{EdgThmHold}, respectively.

\section{Background}
\label{ScBack}
\subsection{A sequential RPF theorem}\label{ScRPF}
For all $j\in\bbN$ and $z\in\bbC$,
let $R_z^{(j)}$ the operator given by
\[
R_z^{(j)} g(x)=\bbE[g(X_{j+1})e^{zf_j(X_j,X_{j+1})}|X_j=x]=R_j(e^{zf_j(x,\cdot)}g)(x)
\]
where $g:\cX_{j+1}\to\bbR$ is a bounded function. Denote by $B_j$ the space of bounded functions on $\cX_j$, equipped with the supremum norm $\|\cdot\|_\infty$.
For every integer $j\geq1$, $n\in\bbN$ and $z\in\bbC$ consider the
$n$-th order iterates $R_z^{j,n}:B_{j+n}\to B_j$
given by 
\begin{equation}\label{Int Op oter 1}
R_z^{j,n}=R_z^{(j)}\circ R_z^{(j+1)}\circ\cdots\circ R_z^{(j+n-1)}.
\end{equation}
We have the following.

\begin{theorem}\label{RPF}
There exists a complex 
neighborhood $U$ of $0$ which depends only on  
$\DS \|f\|_\infty:=\sup_j\sup|f_j|$ and $\ve_0$ (from the definition of the uniform ellipticity)
so that 
for any $z\in U$ and an integer
 $j\geq1$ there exists a  triplet
$\la_j(z)$, $h_j^{(z)}$ and $\nu_j^{(z)}$ consisting of a nonzero complex number
$\la_j(z)$, a  complex function $h_j^{(z)}\in B_{j}$ and a 
continuous linear functional $\nu_j^{(z)}\in B_j^*$ satisfying that $\nu_j^{(z)}(\textbf{1})=1$, $\nu_j^{(z)}(h_j^{(z)})=1$ and
\[
R_z^{(j)}h_{j+1}^{(z)}=\la_j(z)h_j^{(z)}
,\,\,\text{ and }\,\,(R_z^{(j)})^*\nu_{j}^{(z)}=\la_j(z)\nu_{j+1}^{(z)}
\]
where $(R_z^{(j)})^*:B_{j}^*\to B_{j+1}^*$ is the dual operator of $R_j^{(z)}$ and $B_j^*$ is the dual space of the Banach space $B_j$.
When $z=t\in\bbR$ then  $h_j^{(t)}$ is strictly positive, $\nu_j^{(t)}$ is a probability measure and there exist constants $a,b>0$, which depend only on $\|f\|_\infty$ and $\ve_0$ so that $\la_j(t)\in[a,b]$ and $h_j^{(t)}\geq a$. When $t=0$ we have $\la_j(0)=1$ and $h_j^{(0)}=\textbf{1}$.

Moreover, this triplet is analytic and uniformly bounded.
Namely, the maps 
\[
\la_j(\cdot):U\to\bbC,\,\, h_j^{(\cdot)}:U\to B_j\,\,\text{ and }\,
\nu_j^{(\cdot)}:U\to B_j^*
\]
are analytic, where $B_j^*$ is the dual space of $B_j$,
and there exists a constant $C>0$ so that
\begin{equation}\label{UnifBound.1}
\max\Big(\sup_{z\in U}|\la_j(z)|,\, 
\sup_{z\in U}\|h_j^{(z)}\|_{\infty},\, \sup_{z\in U}
\|\nu^{(z)}_j\|_{\infty}\Big)\leq C
\end{equation}
where $\|\nu\|_{\infty}$ is the 
operator norm of a linear functional $\nu:B_j\to\bbC$.

Furthermore, there exist  constants $C>0$  and $\del\in(0,1)$ such that for all
$n\geq1$, $j\in\bbN$, $z\in U$ and $q\in B_{j+n}$,
\begin{equation}\label{Exp Conv final.0.1.1}
\bigg\|\frac{R_z^{j,n}q}{\la_{j,n}(z)}
-\big(\nu_{j+n}^{(z)}(q)\big)h_j^{(z)}\bigg\|_\infty\leq\\
C\|q\|_\infty\cdot \del^n
\end{equation}
and 
\begin{equation}\label{Exp Conv final dual.0.1.1}
\bigg\|\frac{(R_z^{j,n})^*\mu}{\la_{j,n}(z)}
-\big(\mu h_j^{(z)}\big)\nu_{j+n}^{(z)}\bigg\|_\infty\leq \\
C \|\mu\|_\infty\cdot \del^n
\end{equation} 
where $\DS \la_{j,n}(z)=\prod_{k=0}^{n-1}\la_{j+k}(z)$.
\end{theorem}
The proof of Theorem \ref{RPF} was given in \cite[Ch.4\&6]{HK} by a successive application of the complex projective contraction.%  and their dual cones. 
We remark that the arguments in \cite[Ch.4\&6]{HK} formally require us to have a two sided sequence of operators, and in order to overcome this technical difficulty, for $j\leq 0$  we define
 $\cX_j=\cX_1$ and $R_z^{(j)}g(x)=\bbE[g(X_1)]$. This amount to taking independent copies $\{Z_j: j\leq 0\}$ of $X_1$, setting $X_j=Z_j$ for $j\leq 0$ and $f_j=0$. %It is clear that Theorem \ref{RPF0} holds now also for non-positive $j$. 
 In fact, in \cite[Ch.4\&6]{HK} the setup included random operators $R_{z}^{(j)}=R_z^{\te^j\om}$, when $\om\in\Om$ and $(\Om,\cF,P,\te)$ is some invertible measure preserving system, which is not necessarily ergodic. The main reason for considering random operators in \cite{HK}, and not just a sequence of operators, is that the random Ruelle-Perron-Frobenius (RPF) theorem was needed in the proof of the local CLT from \cite[Ch. 2]{HK}, where random operators arise after a certain conditioning argument. The measurability of the resulting RPF triplets $\la_\om(z), h_\om^{(z)}, \nu_\om^{(z)}$ as functions of $\om$ played an important rule in that proof, which lead to consider a more general steup of random operators in \cite[Ch. 4]{HK}, for which there is meaning to such  measurability. However, in our purely sequential setup such measurability issues do not arises, and thus we can just repeat the arguments from \cite[Ch. 4]{HK}  pertaining to a fixed $\omega$
 and ignore the ones addressing measurability.

\begin{remark}\label{CentRem}
In the proof of the Berry-Esseen theorem and the Edgeworth expansions it will be convenient to assume that $a_n:=\bbE[f_n(X_n,X_{n+1})]=0$. This amount to replacing $f_n$ with $f_n-a_n$, and hence to replacing $R_z^j$ with $e^{-a_jz}R_z^j$  and replacing 
$\la_{j}(z)$ with $e^{-za_j}\la_{j}(z)$. 
\end{remark}

\subsection{The structure constants}
As it was mentioned in the introduction a new feature of our work is that we do not make any assumptions
on how slow variance of $S_N$ grows. 
In this section we recall a few results from \cite{DS} which
provide some geometric control on the variance.
 \smallskip

By a {\em random hexagon based at $n$} we mean a tuple 
$$
P_n=(\mathscr X_{n-2},\mathscr X_{n-1}, \mathscr X_n;  \mathscr Y_{n-1},  \mathscr Y_{n},  \mathscr Y_{n+1})
$$
where $(\mathscr X_{n-2},\mathscr X_{n-1})$ and $( \mathscr Y_n, \mathscr Y_{n+1})$ are independent,   $(\mathscr X_{n-2},\mathscr X_{n-1})$ and $(X_{n-2},X_{n-1})$
 are equality distributed, $( \mathscr Y_n, \mathscr Y_{n+1})$ and $(X_n,X_{n+1})$ are equality distributed and $\mathscr X_n$ and $ \mathscr Y_{n-1}$ are conditionally independent given the previous choices and are sampled according to the {\em bridge distributions}
$$
\bbP(\mathscr X_n\in E|\mathscr X_{n-1}=x_{n-1}, \mathscr Y_{n+1}=y_{n+1})=\bbP(X_n\in E|X_{n-1}=x_{n-1},X_{n+1}=y_{n+1})
$$
and 
$$
\bbP( \mathscr Y_{n-1}\in E| \mathscr X_{n-2}=x_{n-2}, \mathscr Y_{n}=y_{n})=\bbP(X_{n-1}\in E|X_{n-2}=x_{n-2},X_{n}=y_{n}).
$$ 

The balance $\Gamma(P_n)$ of the hexagon is given by  
$$\Gamma(P_n)=f_{n-2}(\mathscr X_{n-2}, \mathscr X_{n-1})+f_{n-1}(\mathscr X_{n-1},\mathscr X_n)+f_n(\mathscr X_n,\mathscr Y_{n+1})$$
$$-
f_{n-2}(\mathscr X_{n-2},\mathscr Y_{n-1})-f_{n-1}(\mathscr Y_{n-1},\mathscr Y_n)-f_n(\mathscr Y_n,\mathscr Y_{n+1}).$$

Next, let
\begin{equation}\label{u n def}
u_n^2=\bbE[\Gamma(P_n)^2].
\end{equation}
\begin{theorem}[\cite{DS}, Theorem 2.1]\label{Thm2.1}
There exist positive constants $C_1,C_2,C_3,C_4$ so that for any $m\geq0$ and $N\geq3$,
\begin{equation}\label{U N V N relation}
 C_1\sum_{n=m+3}^{m+N}u_n^2-C_2\leq V_N=\mathrm{Var}(S_N-S_m)\leq C_3\sum_{n=m+3}^{m+N}u_n^2+C_4.
\end{equation}
\end{theorem}

It turns out that the hexagon process also allows to control the characteristic function of $S_N.$
Denote
\begin{equation}\label{d n def}
d_n(\xi)^2=\bbE[|e^{i\xi\Gamma(P_n)}-1|^2]=4\bbE[\sin^2(\xi\Gamma(P_n)/2)], \quad
D_N(\xi)=\sum_{n=1}^N d_n^2(\xi).
\end{equation}

\begin{lemma}[\cite{DS}, eq. (4.2.6)]
\label{Lm2.2}
There are constants $C, c>0$ so that for each $N$ and $\xi\in\bbR$,
the characteristic function $\Phi_N(\xi)=\bbE\left(e^{i\xi S_N}\right)$ satisfies
\begin{equation}\label{u n rel}
\left|\Phi_N(\xi)\right|\leq C e^{-c D_N(\xi)}. 
\end{equation}
\end{lemma}
 %For readers' convenience let us recall the definition of $P_n$ and $\Gamma(P_n)$. First $P_n$ is  a hexagon of the form 

\subsection{Mixing and moment estimates}
%In \cite[Proposition 1.1]{DS} the mixing properties of $\{\xi_n\}$ were investigated. We formulate here only the exponential decay of correlation proved there.

 Next we discuss the mixing properties of $\{X_n\}.$

\begin{lemma}[Proposition  1.11 (2), \cite{DS}]\label{MixLemma}
There exist $\del\in(0,1)$ and $A>0$ so that for all $n,k\in\bbN$ we have 
$$
\left|\text{\rm Cov}\big(f_n(X_n,X_{n+1}), f_{n+k}(X_{n+k},X_{n+k+1})\big)\right|\leq A\del^k.
$$
\end{lemma}

Next, for each $j$ and $n$ consider the random variable $S_{j,n}$
given by 
\begin{equation}\label{Snj}
S_{j,n}=\sum_{k=j}^{j+n-1}f_k(X_k,X_{k+1}).
\end{equation}
Then $S_{1,n}=S_n$.
\begin{lemma}[Lemma 2.16, \cite{DS}]\label{MomLem}
For every integer $p\geq1$ there are constant $C_p,R_p>0$ so that for all $j$ and $n$,
\[
\left|\bbE \left[\big(S_{j,n}-\bbE(S_{j,n})\big)^p\right]\right|\leq R_p+C_p\Big(\text{\rm Var}(S_{j,n})\Big)^{[p/2]}.
\]
\end{lemma}
%We will also need the following result
%\begin{lemma}[Lemma \cite{DS}]\label{CharDecLem}
%There exist  constants $C_3,c_3>0$ and $\del_0>0$ so that for any $n$ and $t\in[-\del_0,\del_0]$ we have
%\[
%|\mathbb E[e^{it S_n}]|\leq C_3e^{-c_3t^2\sig_n^2}.
%\]
%\end{lemma}

\section{Edgewoth expansions under logarithmic growth assumptions}\label{Sec4}
\subsection{The Edgewoth polynomials}
Let $S$ be a random variable with finite moments of all orders. We recall that 
 the $k$-th cumulant of $S$ is given by
\[
\Gam_k(S)=\frac1{i^k}\frac{d^k}{dt^k}\big(\ln\bbE[e^{itS}]\big)\big|_{t=0}.
\]
Note that $\Gamma_k(aS)=a^k\Gamma_k(S)$ for every $a\in\bbR$. Moreover,
 $\Gamma_1(S)=\bbE[S]$, $\Gamma_2(S)=\text{Var}(S)$ and for $k\geq3$
  by (1.34) in \cite{SaulStat},  we have
\begin{equation}\label{Cumulants formula}
\Gamma_k(S)=\sum_{v=1}^k\frac{(-1)^{v-1}}{v}\sum_{k_1+...+k_v=k}\frac{k!}{k_1!k_2!\cdots k_v!}\al_{k_1}\al_{k_2}\cdots\al_{k_v}
\end{equation}
where $\al_m=\al_m(S)=\mathbb E[S^m]$ (this formula is a consequence of the Taylor expansion of the function $\ln(1+z)$). 

The cumulants of order $k\geq3$ measure the distance of the distribution of $\hat S=\big(S-\bbE[S]\big)/\sig$,  from the standard normal distribution, where $\sig=\sqrt{\text{Var}(S)}$, assuming of course that $\sig>0$. We have $\Gam_k(S)=0$ for all $k\geq3$ if and only if $\hat S$ is standard normal, and we refer to \cite{SaulStat} for  conditions on $\Gamma_k(\hat S)$
which insure that the distribution function of $\hat S$ is close to the standard normal distribution function in the uniform metric. We also refer to \cite{Bar, RinRot}  
 for expansions of expectations of smooth functions of
$\hat S$ which involve growth properties  of cumulants.

Next, let us assume that $\bbE[S]=0$ and $\sig^2=\bbE[S^2]>0$.
Consider the function 
$$
\Lambda(t;S)=\ln\bbE[e^{it S/\sig}]+t^2/2.
$$
Then $\Lambda(0;S)=0$, $\Lambda_n'(0;S)=\bbE[S]=0$,  $\Lambda_n''(0;S)=\bbE[S^2]/\sig^2-1=0$, and for $k\geq3$
we have 
$$\Lambda^{(k)}(0):=\frac{d^k}{dt^k}\Lambda(t;S)\big|_{t=0}={i^k}\Gamma_k(S)\sig^{-k}.$$ 
Thus, the $k$-th Taylor polynomial of $\Lambda(t;S)$ is given by
$$
\cP_{k}(t;S)=\sum_{j=3}^{k}\frac{i^j\Gamma_{j}(S)}{j!\sig^j}t^j=
\sum_{j=3}^k i^ja_j(S)\sig^{-(j-2)}t^j.
$$
where\footnote{The reason we divide $\Gamma_{j}(S)$ by $\sig^2$ is that under suitable restrictions on $S$, the quantities $|\Gamma_{j}(S)\sig^{-2}|$ will be bounded by some constant not depending on $S$ (see next section). This will be the case when $S=S_n$, for which the latter quantities will be bounded in $n$. Here $S_n$ are the sums considered in Section \ref{Main}.} $a_j(S)=\frac{\Gamma_{j}(S)}{j!\sig^2}$.
Let us consider the formal power series 
$$\Gamma(t;S)=\sum_{j\geq 3}\frac{i^j\Gamma_j(S)}{j!\sig^j}t^j=\sum_{j\geq 3}{i^j}a_j(S)\sig^{-(j-2)}t^j,$$
where $a_j(S)$ is viewed as a variable independent of $\sig$.
%Then, the expansion of the above power series in the variable $\sig^{-1}$ is
 This leads to the following formal series
$$
\exp(\Gamma(t;S))=1+\sum_{j\geq1}\frac{i^j\Gamma(t;S)^j}{j!}
=1+\sum_{j\geq1}\sig^{-j}A_{j}(t;S)
$$
where $A_{j}(t;S)$ is the polynomial given by 
$$
A_{j}(t;S)=\sum_{m=1}^{j}\frac{1}{m!}\sum_{k_1,...,k_m\in \cA_{j,m}}\prod_{u=1}^{m}i^{k_i}a_{k_i}(S)t^{j+2m}
$$
and  $\cA_{j,m}$ is the set of all $m$-tuples $(k_1,...,k_m)$ of integers such that 
$$k_i\geq3\quad\text{and}\quad 
\sum_{i}k_i=2m+j.$$

\begin{definition}\label{EdgPdef}
The $j$-th Edgewoth polynomial $S$ is the unique polynomial $P_{j}(t;S)$ so that the Fourier transform of $\phi(t)P_{j}(t;S)$ is $e^{-t^2/2}A_{j}(t;S)$, where $\phi(t)$ is the standard normal density.
\end{definition}
Notice that the polynomials $A_j(t;S)$ and $P_{j}(t;S)$ depend on $S$ only through the first $3j$ moments. Note also that $A_j(0;S)=0$ for all $j$.

\begin{remark}\label{CompRem}
In order to compute $A_{j}(t;S)$ for $j\leq k$ it is enough to expand 
$e^{\cP_{k+2}(t;S)}$ 
to a power series and represent it in the form 
$\DS 1+\sum_{j\geq1}\sig_n^{-j}\tilde A_{j}(t;S)$. Indeed, it follows that $\tilde A_{j}(t;S)=A_{j}(t;S)$ for all $j\leq k$ since 
$$\Gamma(t,S)-\cP_{k+2}(t;S)=\sig^{-(k+1)}\sum_{j=k+3}^{\infty}i^ja_j(S)\sig^{-(j-k-3)}t^j. $$
 Thus, to compute $A_{j}(t;S)$, $j\leq k$ we first write 
$$
e^{\cP_{k+2}(t;S)}=1+\sum_{j=1}^\infty\frac{\cP_{k+2}(t;S)^j}{j!}.
$$
Now, since $\cP_{k}(t;S)$ has a  factor\footnote{Recall that $a_j(S)$ are viewed as constants.} $\sig^{-1}$, we can compute $A_{j}(t;S)$, $j\leq k$ by considering only the first $k$ summands
$$
1+\sum_{j=1}^{k}\frac{\cP_{k+2}(t;S)^j}{j!}.
$$
After writing the above  expression in the form 
$\DS 1+\sum_{j=1}^{\infty}\sig^{-j}\bar A_{j,k}(t;S)$ (this is a finite sum) we have $A_{j}(t;S)=\bar A_{j,k}(t;S)$ for all $j\leq k$.

 In particular $\DS  \cP_3(t; S)=\frac{i^3a_3(S)t^3}{6\sig}=\frac{A_1(t; S)}{\sig}$, whence
$$ P_1(t; S)=\frac{a_3(S)}{6} (t^3-3t)=\frac{\bbE[(S-\bbE[S])^3]}{6\sig^2}(t^3-3t)$$
where we have used that the transform Fourier of $(t^3-3t)\phi(t)$ is $i^3e^{-\frac12\xi^2}\xi^3$.

 \end{remark}

\subsection{A Berry-Esseen theorem and Edgeworth expansions via decay of characteristic functions}
Let $W_n$ be a sequence of centered random variables so that $\DS \lim_{n\to\infty}\text{Var}(W_n)=\infty$.
Let us set $$\Gamma_n(t)=\Gamma(t;W_n),\quad
\Lambda_n(t)=\Lambda(t;W_n),\quad A_{j,n}(t)=A_j(t;W_n),  
\quad P_{j,n}(t)=P_j(t;W_n). $$
Let us also set $\sig_n=\sqrt{\text{Var}(W_n)}$.
We will prove here Edgeworth expansions under the following logarithmic growth assumptions.
\begin{assumption}\label{GrowAssum}
For some $k\geq3$, for all $j\leq k$ there exist constants $C_j,\ve_j>0$ so that
\begin{equation}\label{GAs}
\sup_{t\in[-\ve_j\sig_n,\ve_j\sig_n]}|\Lambda_n^{(j)}(t)|\leq C_j\sig_n^{-(j-2)}.
\end{equation}
\end{assumption}
Note that under Assumption \ref{GrowAssum} the polynomials $A_{j,n}$ and $P_{j,n}$, $j\leq k$ have bounded coefficients (for that it is enough to only consider $t=0$). For $t=0$ 
conditions of the form
 $|\Lambda_n^{(j)}(0)|=|\Gamma_j(W_n/\sig_n)|\leq (j!)^{1+\gamma}\sig_n^{-(j-2)},\gamma\geq0$ 
appear in literature \cite{Dor,DN, SaulStat} in the context of moderate deviations and related results (see also references therein).

The relevance of Assumption \ref{GrowAssum} stems from the following facts.
 
\begin{proposition}\label{BE prop}
Let Assumption \ref{GrowAssum} hold with $k=3$. Then
there exists a constant $C>0$ so that for every $n\geq1$
we have 
$$
\sup_{t\in\bbR}\left|\bbP(W_n/\sig_n\leq t)-\Phi(t)\right|\leq C\sig_n^{-1}
$$
where $\Phi$ is the standard normal distribution and density function.
\end{proposition}

\begin{proposition}\label{EdgeProp}
Let $r\geq1$ be an integer.
Let Assumption  \ref{GrowAssum} hold with $k=r+3$.
Suppose also that for every $B>0$ and  all 
$\del>0$ %small enough,
\begin{equation}\label{Cond2.0}
\int_{\del\leq |x|\leq B\sig_n^{r-1}}|\bbE(e^{ix W_n})/x|dx=o(\sig_n^{-r}).
\end{equation}
Then
\begin{equation}\label{genExp}
\sup_{t}\left|\bbP(W_n/\sig_n\leq t)-\Phi(t)-\sum_{j=1}^r\sig_n^{-j} P_{j,n}(t)\phi(t)\right|=o(\sig_n^{-r})
\end{equation}
where $\Phi$ and $\phi$ are the standard normal distribution and density function, respectively.
\end{proposition}

\subsection{Auxillary estimates.}
Here we present several technical estimates needed in the proofs of 
Propositions \ref{BE prop} and \ref{EdgeProp}.

We need two lemmata.

\begin{lemma}\label{Le2}
Let $k\geq3$ be an integer and let 	Assumption \ref{GrowAssum} 
hold with this $k$. 
Then there exist constants $\del_k,B_k>0$ so that for every real $t\in[-\sig_n,\sig_n]$,
$$|\cP_{k,n}(t)|\leq B_k\sig_n^{-1}|t|^3=B_kt^2|t/\sig_n|.$$
Therefore, there is a constant $\del_k>0$ so that for every $t\in[-\del_k\sig_n,\del_k\sig_n]$,
$$|e^{\cP_{k,n}(t)}|\leq e^{t^2/4}.$$
\end{lemma}

\begin{lemma}\label{Le1}
Let Assumption \ref{GrowAssum} hold with $k=3$.
Then there exist
$\del_0>0$ and $\al\in(0,1/2)$ so that for every  real
 $t$ such that $|t/\sig_n|\leq \del_0$ we have $|e^{\Lambda_n(t)}|\leq e^{\al t^2}$.
\end{lemma}

\begin{proof}[Proof of Lemmas \ref{Le2} and \ref{Le1}]
Let us first prove Lemma \ref{Le2}.
By taking $t=0$ in \eqref{GAs} and using that $\Gamma_j(aW)=a^j W$
 we have $|\Gamma_j(W_n)|\leq C_j\sig_n^2$. 
 Thus, if $|t/\sig_n|<1$ then with  $\DS A_k=\max_{3\leq j\leq k}C_j$ and $B_k=kA_k$ we have
$$
|\cP_{k,n}(t)|\leq \sum_{j=3}^k\frac{|\Gamma_j(W_n)|}{j!\sig_n^j}|t|^j
\leq A_kt^2\sum_{j=3}^k|t/\sig_n|^{j-2}/j!\leq A_k t^2\sum_{j=3}^k |t/\sig_n|\leq B_k|t|^3\sig_n^{-1}.
$$
Hence, if $|t/\sig_n|\leq\frac{1}{4B_k}:=\del_k$ then $|\cP_{k,n}(t)|\leq t^2/4$ and so 
$$
\left|e^{\cP_{k,n}(t)}\right|\leq e^{t^2/4}.
$$
Finally, to prove Lemma \ref{Le1}, using that the second Taylor polynomial $\cP_{2,n}(t)$ of $\Lambda_n$ around the origin vanishes, we can write write $\Lambda_n(t)=\cP_{2,n}(t)+\cR_{2,n}(t)=\cR_{2,n}(t)$,  where $\cR_{2,n}(t)$ is the Taylor remainder of order $2$ around the origin. Then by 
 the Lagrange form of the Taylor remainder we can write $\cR_{2,n}(t)=\frac{t^3\Lambda_n'''(t_1)}{3!}$ for some $t_1$ such that $|t_1|\leq|t|$. Therefore, 
by  Assumption \ref{GrowAssum} we have $$\cR_{2,n}(t)=O(t^3/\sig_n)=t^2O(|t/\sig_n|),\,  \text{if }\, |t|\leq\ve_1.$$
Thus when $|t/\sig_n|$ is small enough $|\Lambda_n(t)|=|\cR_{2,n}(t)|<t^2/3$, and Lemma \ref{Le1} follows with $\al=1/3$.
\end{proof}

 Combining Lemma \ref{Le1} with Proposition \ref{VerifProp} proven in Section \ref{ScMCCum}
we  recover the following result, which was proved in \cite[Theorem 6.1]{DS} for the uniformly elliptic Markov chains considered in this paper.
\begin{corollary}
Under assumption \ref{GrowAssum} with $k=3$ there exist   constants $c>0$ and $\del>0$ so that for every natural $n$ and $t\in[-\del,\del]$ we have
\[
|\mathbb E[e^{it W_n}]|\leq e^{-ct^2\sig_n^2}.
\]
\end{corollary}
\vskip3mm

 The key step in estimating the rate of convergence for the CLT is the following.

\begin{proposition}\label{PropEdg}
Let $r\geq0$ be an integer and let
 Assumption \ref{GrowAssum} hold with $k=r+3$. Then  there is a constant $\del_r>0$ such that 
$$
\int_{-\del_r \sig_n}^{\del_r \sig_n}\left|\frac{\bbE[e^{itW_{n}/\sig_n}]-e^{-t^2/2}(1+Q_{r,n}(t))}{|t|}\right|dt=O(\sig_n^{-r-1})
$$
where for $r=0$ we set $Q_{0,n}(t)=0$ and for $r\geq1$,
\[
Q_{r,n}(t)=\sum_{j=1}^{r}\sig_n^{-j}A_{j,n}(t)
\]

\end{proposition}

\begin{proof}
Write 
\begin{equation}\label{Base}
\bbE[e^{itW_n/\sig_n}]=e^{-t^2/2}e^{\Lambda_n(t)}=e^{-t^2/2}e^{\cP_{r+2,n}(t)+\cR_{r+2,n}(t)}
\end{equation}
where $\cR_{r+2,n}(t)$ is the Taylor remainder of order $r+2$ around $0$. 
Using the Lagrange form of Taylor remainders together with
Assumption \ref{GrowAssum} we get that
\begin{equation}\label{Zero}
\cR_{r+2,n}(t)=O(t^{r+3}\sig_n^{-(r+1)}).
\end{equation}
Next, by the mean value theorem and Lemmas \ref{Le2} and \ref{Le1} there are constants $\del_r>0$, $C_0>0$ and $b_0\in(0,1/2)$ so that if $|t/\sig_n|\leq\del_r$ then
\begin{equation}\label{One}
\left|e^{\Lambda_n(t)}-e^{\cP_{r+2,n}(t)}\right|\leq C_0e^{b_0 t^2}|\cR_{r+2,n}(t)|.
\end{equation}
Moreover, by Lemma \ref{Le2} and the Lagrange form of Taylor remainders,
\begin{equation}\label{Two}
\left|e^{\cP_{r+2,n}(t)}-
\left(1+\sum_{j=1}^{r}\frac{\cP_{{j+2},n}(t)^j}{j!}\right)\right|\leq 
D_{r}e^{b_0t^2}\sig_n^{-(r+1)}|t|^{3(r+2)}
\end{equation}
where $D_{r}>0$ is some constant 
(when $r=0$ then the left hand side vanishes since $\cP_{2,n}(t)=0$).
Combining \eqref{Base}, \eqref{Zero}, \eqref{One} and \eqref{Two}, for every real $t$ so that $|t/\sig_n|\leq \del_{r}$ we have 
$$
\left|\bbE[e^{itW_n/\sig_n}]-e^{-t^2/2}\left(1+\sum_{j=1}^{r}\frac{\cP_{{j+2},n}(t)^j}{j!}\right)\right|
\leq Ce^{-ct^2}\sig_n^{-(r+1)}\max\left(|t|, |t|^{(r+3)(r+2)}\right)
$$
where $c=1/2-b_0>0$.
Next, by Remark \ref{CompRem}, we have
$$\sum_{j=1}^r\frac{\cP_{r+2,n}(t)^j}{j!}=Q_{r,n}(t)+
\max( |t|, |t|^{r(r+2)})
O(\sig_n^{-r-1})$$
where the term  $\max(|t|, |t|^{r(r+2)})O(\sig_n^{-r-1})$ comes from the terms which include powers of $\sig_n^{-1}$ larger than $r$ 
(when $r=0$ both the left hand side and $Q_{r,n}(t)$ equal $0$).
We conclude that
$$
\left|\bbE[e^{itW_n/\sig_n}]-e^{-t^2/2}(1+Q_{r,n}(t))\right|\leq Ce^{-ct^2}\sig_n^{-(r+1)}
\max\left(|t|, |t|^{(r+3)(r+2)}\right).
$$
Therefore,
$$
\int_{-\del_r \sig_n}^{\del_r \sig_n}\left|\frac{\bbE[e^{itW_{n}/\sig_n}]-e^{-t^2/2}(1+Q_{r,n}(t))}{|t|}\right|dt
$$
$$
\leq C\sig_n^{-(r+1)}\int_{-\infty}^{\infty}e^{-ct^2}\left(1+|t|^{(r+3)(r+2)-1}\right)dt\leq C'\sig_n^{-(r+1)}
$$
completing  the proof of the proposition.
\end{proof}

\subsection{Proofs of Propositions \ref{BE prop} and \ref{EdgeProp}}
\begin{proof}[Proof of Proposition \ref{BE prop}]
The first step in the proof is quite standard. 
We use  generalized Esseen inequality \cite[\S XVI.3]{Feller}. 
Let $F:\bbR\to\bbR$ be a probability distribution function and $G:\bbR\to\bbR$ be a differential function with bounded derivative so that $G(-\infty)=0$. Let $f(t)=\int e^{itx}dF(x)$ and $g(t)=\int e^{it x}dG(x)$ be the corresponding Fourier transforms. Then 
for every $T>0$ we have
\begin{equation}\label{LZ}
\sup_{x\in\bbR}|F(x)-G(x)|\leq 2 \int_{-T}^{T}\left|\frac{f(t)-g(t)}{t}\right|dt+\frac{24\|G'\|_\infty}{\pi T}.
\end{equation}

Taking 
$F$ to be the distribution of $W_n/\sig_n$, $G$ to be the standard normal distribution and $T_n= \del_1\sig_n$ 
 where $\del_1$ comes from Lemma \ref{Le2} we conclude that Proposition \ref{BE prop} will follow if we prove that
\begin{equation}\label{BE Need}
\int_{-\del_1\sig_n}^{\del_1\sig_n}\left|\frac{\bbE[e^{it W_n/\sig_n}]-e^{-t^2/2}}{t}\right|dt\leq C\sig_n^{-1}
\end{equation}
for some constant $C$. Finally, \eqref{BE Need} follows from Proposition \ref{PropEdg} with  $r=0$.
\end{proof}

\begin{proof}[Proof of Proposition \ref{EdgeProp}]
Relying on Proposition \ref{PropEdg}, the proof proceeds essentially in the same way as  \cite{Ess,Feller}.
We provide the details  for readers' convenience.

Let $F=F_n$ be the distribution function of $W_n/\sig_n$, and $G=G_{n,r}$ be the function whose Fourier transform is $e^{-t^2/2}(1+Q_{n,r}(t))$, where $Q_{n,r}$ comes from Proposition \ref{PropEdg}. 
Then $G_{n,r}$ has the form 
$$
G_{n,r}(t)=\Phi(t)+\sum_{j=1}^r\sig_n^{-j}P_{j,n}(t){\phi(t)}
$$
where $P_{j,n}$'s  are the Edgeworth polynomials of $W_n$.

Let $\ve>0$ and $B=1/\ve$.
Applying \eqref{LZ} with $F=F_n$, $G=G_n$ and $T=B\sig_n^r$ we obtain 
$$
\sup_{t}\left|P(W_n/\sig_n\leq t)-\Phi(t)-\sum_{j=1}^r\sig_n^{-j} P_{j,n}(t)\phi(t)\right|\leq
I_1+I_2+I_3+O(\ve)\sig_n^{-r}
$$ 
where for $\del$ small enough
$$
I_1=\int_{-\del\sig_n}^{\del \sig_n}\left|\frac{\bbE[e^{it W_n/\sig_n}]-e^{-t^2/2}(1+Q_{r,n}(t))}{t}\right|dt
$$
$$
I_2=\int_{\del\sig_n\leq |t|\leq B\sig_n^r}\left|\frac{\bbE[e^{it W_n/\sig_n}]}{t}\right|dt,
\quad
I_3=\int_{|t|\geq\sig_n\del}e^{-t^2/2}\left|\frac{1+Q_{r,n}(t)}{t}\right|dt.
$$
By Proposition \ref{PropEdg} we have $I_1=o(\sig_n^{-r})$,
 \eqref{Cond2.0}
gives that $I_2=o(\sig_n^{-r})$, while  $I_3=O(e^{-c\sig_n^2})$ for some $c>0$ since $Q_{r,n}$ is a polynomial with bounded coefficients and degree depending only on $r$.
\end{proof}

\section{Application to uniformly elliptic inhomogeneous Markov chains}
\label{ScMCCum}
\subsection{Verification of Assumption \ref{GrowAssum}}
 In this section we consider uniformly bounded additive functional $S_N$ of a Markov chain $X_n$
which satisfies \eqref{DUpper} and \eqref{DLower}.
We prove the following.
\begin{proposition}\label{VerifProp}
The sequence of random variables $S_n$ verifies Assumption \ref{GrowAssum} for every $k$, namely, if $\Lambda_n(t)=\ln\bbE[e^{it S_n/\sig_n}]+t^2/2$ then for every $k\geq3$ there exist constants $\del_k,C_k>0$ so that for all $n$, 
$$
\sup_{t\in[-\sig_n\del_k, \sig_n\del_k]}|\Lambda_n^{(k)}(t)|\leq C_k\sig_n^{-(k-2)}.
$$
\end{proposition}
The proof of Proposition \ref{VerifProp} is based on the construction of sequential pressure functions, as described in the following section.
\begin{remark}\label{MCRem}
In \cite[Theorem 4.26]{SaulStat} the authors show that 
 if $\DS S_n=\sum_{j=1}^n Y_j/\sig_n$,  and $\{Y_j\}$ is an
 exponentially fast $\phi$-mixing uniformly bounded  centered Markov chain, 
 such that $\text{Var}(Y_j)$ is bounded away from $0$ then there is a constant $C$ such that 
 for all $m\in \mathbb{N}$
%\begin{equation}\label{LambdaRA}
$\DS |\Gamma_j(S_n/\sig_n)|\leq C^mm!\sig_n^{-(m-2)}$.
%\end{equation} % $m=3,4,5,...$ where $C$ is some constant, 
It follows that  the function $\Lambda_n$ is real analytic and, hence,
Assumption \ref{GrowAssum}  holds for every $k$.
   By \cite[Proposition 1.22]{DS}, the Markov chains $\{X_n\}$ considered in this paper are also exponentially fast $\phi$-mixing, however, we consider functionals $Y_n=f_n(X_n,X_{n+1})$ whose variance can be small, and so  Proposition \ref{VerifProp} cannot be derived from \cite[Theorem 4.26]{SaulStat} despite the related setup.
\end{remark}

\subsection{The sequential pressure function. Definition and basic properties}

 Recall Theorem \ref{RPF}.
For every $j\geq1$, 
denote by $\mu_j$ the distribution of $X_j$ (which is a probability measure on 
$\cX_j$).
Recall that $\la_j(z)$ is uniformly bounded in $j$ and $\la_j(0)=1$.
Let $\Pi_j(z)$ denote 
 the analytic branch  of the logarithms of $\la_j(z)$, such that  $\Pi_j(0)=0$. 
We call $\Pi_j(z)$ the {\em sequential pressure functions}.
 Then  
\begin{equation}\label{Press bound}
\sup_j\sup_{|z|\leq s_0}|\Pi_j(z)|\leq c_0
\end{equation}
where $s_0$ and $c_0$  are some positive constants. We note that all the derivatives of $\Pi_j$ at $z=0$ are real numbers, since the function $\la_j(z)$ is positive for real $z$'s.

\begin{remark}\label{CentRem2}
By Remark \ref{CentRem}, upon replacing $f_n$ with $f_n-\bbE[f_n(X_n,X_{n+1})]$, the resulting pressure function becomes $\Pi_{j}(z)-\bbE[f_n(X_n,X_{n+1})]z$. This has no affect on the value of the pressure function at $z=0$ and on the derivatives of it of any order larger than $1$. Thus, it will essentially make no difference in the following arguments if we have already centralized $f_n$ or not. 
\end{remark}

Let $j,n$ be positive integers.
Set 
\[
\Gamma_{j,n}(z)=\ln\mathbb E[e^{zS_{j,n}}], \quad 
\Pi_{j,n}(z)=\sum_{s=j}^{j+n-1}\Pi_s(z)
\]
where $S_{j,n}$ is defined in \eqref{Snj}.

\begin{lemma}\label{LemmaExAp}
There is a constant $a>0$ with the following property:
for every integer $k\geq 0$ there exists
 $c_k>0$ such that 
 for each $j, n$
for all complex $z$ so that $|z|\leq a$ we have 
\begin{equation}\label{ExpApprox}
\big|\Gamma_{j,n}^{(k)}(z)-\Pi_{j,n}^{(k)}(z)\big|\leq c_k
\end{equation}
where $g^{(k)}(z)$ denotes the $k$-th derivative of a function $g(z).$ 
\end{lemma}
Note that for $k=0,1,2$ and $z=0$ we have
$\Gamma_{j,n}^{'}(0)=\bbE[S_{j,n}]$, $\Gamma_{j,n}^{''}(0)=\text{Var}(S_{j,n})$ while for larger $k$'s $\Gamma_{j,n}^{(k)}(0)$ is just the $k$-th cumulant of $S_{j,n}$. In particular, 
$$\Pi_{1,n}'(0)=\bbE(S_n)+O(1)\quad \textrm{and} \quad \Pi_{1,n}''(0)=\sig_n^2+O(1) . $$
%Note also that (\ref{ExpApprox})  yields that the cumulants barely depend on the initial distribution (indeed, taking $z=0$ yields that the difference between the $k$-th cumulants is bounded in $n$).
 
 \begin{proof}
Since $h_j(0)=\textbf{1}$ and the norms $\|h_j^{(z)}\|_\infty$ are uniformly bounded in $j$ around $0$,
it follows from the Cauchy integral formula that $\frac{\partial h_j}{\partial z}$ is uniformly bounded 
around the origin.
Hence, if $\del_0$ is small enough then for any complex $z$ with $|z|\leq \del_0$ we have 
\begin{equation}\label{LB}
\frac12<\inf_{j}|\mu_1(h_1^{(z)})|.
\end{equation}
Recall that
$\DS
\mathbb E[e^{zS_{j,n}}]=\mu_{j}(R_z^{j,n}\textbf{1}).
$
By (\ref{Exp Conv final.0.1.1}), if  $|z|$ is sufficiently small then for all $j$ and $n$ we have 
\begin{equation}\label{Exp}
\mathbb E[e^{zS_{j,n}}]=e^{\sum_{s=j}^{j+n-1}\Pi_s(z)}(\mu_j(h_j^{(z)})+\del_{j,n}(z))
\end{equation}
where $\del_{j,n}$ is an analytic function so that $|\del_{j,n}(z)|\leq C\del^n$ for some $C>0$ and $\del\in(0,1)$ which do not depend on $j$ and $n$. In fact, since $h_j^{(0)}=\textbf{1}$ we have 
$\del_{j,n}(0)=0$ and so  Cauchy integral formula also implies
$|\del_{j,n}(z)|\leq C|z|\del^n$.
Using \eqref{LB}, we can take
the logarithms of both sides of (\ref{Exp}) and  derive that when $|z|$ is sufficiently small, 
 there is a constant $c_0$ so that 
\begin{equation}\label{ExpApprox}
\big|\Gamma_{j,n}(z)-\Pi_{j,n}(z)\big|\leq c_0
\end{equation} 
Applying the Cauchy integral formula once more we conclude that
for each $k$ there exists a constant $c_k>0$ so that for every $j$ and $n$ we have
\begin{equation}\label{ExpApprox}
\big|\Gamma_{j,n}^{(k)}(z)-\Pi_{j,n}^{(k)}(z)\big|\leq c_k
\end{equation} 
and the lemma follows.
\end{proof}

\subsection{The derivatives of the pressure function around the origin}
%The first ingredient of the proof of Theorem \ref{BE} is the following
%\begin{lemma}\label{Taylor Lemma}
%Let $B_1<B_2$ be constants. Then if $B_1$ is sufficiently large then there are constants
%$D$ and $r_0$ depending only on $B_1$ and $B_2$ so that
%for each $n$ and $j$ and $t$ such that $B_1\leq \text{Var}(S_{j,n})\leq B_2$ and $t\in[-r_0,r_0]$,
%we have
%\[
%\left|\Pi_{j,n}(it)-\left(it\Pi_{j,n}'(0)-\frac{1}{2} t^2\Pi_{j,n}''(0)-\frac{i t^3}{6} \Pi_{j,n}'''(0)%\right)
%\right|\leq  D |t|^4. 
%\] 
%\end{lemma}

Here we prove several useful auxiliary estimates.
\begin{lemma}\label{GenCharLemma}
Let $k\geq2$ be an integer, and let $S$ be a real-valued random variable with finite first $k$ moments. %Let $B_2$ be a constant so that $\|S\|_{L^2}\leq B_2$. 
Let us define $\varphi(t)=\mathbb E\left(e^{itS}\right)$ and $\Lambda(t)=\ln\varphi(t)$. Then there exists a constant $D_k$ which depends only on $k$ so that with 
$\DS r_0=\frac1{2\sqrt{\bbE(S^2)}}$
$$
\sup_{t\in[-r_0,r_0]} |\Lambda^{(k)}(t)|\leq D_k \bbE[|S|^k].
$$
\end{lemma}
\begin{proof}
We first recall that for  the characteristic function $\varphi(t)=\mathbb E\left(e^{itS}\right)$ 
of a random variable  $S$
with finite first $k$  moments and any real $t$ we have
\[
|\varphi(t)-1|\leq |t|\mathbb E[|S|]\leq |t|\mathbb\|S\|_{L^2}
\]
and that
for $j=0,1,2,...,k$ we have
\begin{equation}\label{Bound0} 
|\varphi^{(j)}(t)|\leq %R_k 
\bbE[|S|^j].
\end{equation}

Next, let $\Lambda(t)=\ln\varphi(t)$ and
$\DS r_0=\frac1{2\sqrt{\bbE(S^2)}}$.
 Then $|\varphi(t)|\geq \frac{1}{2}$  for all $t\in[-r_0,r_0]$.
By Fa\'a di Bruno's formula, for every $t\in[-r_0,r_0]$ we have
$$
|\Lambda^{(k)}(t)|
=\left|\sum_{m_1,...,m_{k}}\frac{k!}{\prod_{j=1}^{k}(m_j!(j!)^{m_j})}\cdot\frac1{\varphi(t)^{\sum_{j=1}^k m_j}}\prod_{j=1}^{k}\left((i)^j\bbE[S^je^{it S}]\right)^{m_j}
\right|
$$
where 
$(m_1,...,m_k)$ range over all the $k$-tuples of nonnegative integers such that $\DS \sum_{j}jm_j=k$.
Now the lemma follows from  
\eqref{Bound0} and the H\"older inequality.
% Hence, by (\ref{Bound0}) and the H\"older inequality
%$$ |\Lambda^{(k)}(t)|\leq D_k \bbE[|S|^k] $$
%where $D_k$ is an absolute constant.
\end{proof}

%Next, we will prove the following.
\begin{lemma}\label{Taylor Lemma2}
Fix some integer $k\geq2$ and
let $B_1<B_2$ be constants. Then if $B_1$ is sufficiently large  there are constants
$D$ and $r_0$ depending only on $B_1$, $B_2$ and k so that
for every $t\in[-r_0,r_0]$ and each $ j, n\in\bbN$ such that $B_1\leq \text{Var}(S_{j,n})\leq B_2$,
we have
$$
|\Pi_{j,n}^{(k)}(it)|\leq D.
$$
%\[
%\left|\Pi_{j,n}(it)-\sum_{k=0}^r\frac{\Pi_{j,n}^{(k)}(0)}{k!}(it)^j
%\right|\leq  D |t|^{r+1}. 
%\] 
\end{lemma}
\begin{proof}%Modifty the proof, use the faa di bruno formula

Let $\Lambda_{j,n}(t)=\ln\bbE[e^{it S_{j,n}}]$. Then, in the notation of Lemma \ref{LemmaExAp},
 %we have 
 $\Lambda_{j,n}(t)=\Gamma_{j,n}(it)$.
 Applying Lemma \ref{GenCharLemma} with $S=S_{j,n}$ and using 
\eqref{ExpApprox} and Lemma \ref{MomLem} we obtain that for every $t\in[-r_0,r_0]$ we have
$$
|\Pi_{j,n}^{(k)}(it)|\leq c_{k}+|\Lambda_{j,n}^{(k)}(t)|\leq c_{k}+D_k\bbE[|S_{j,n}|^{k}]
\leq 
c_{k}+C\big(\text{Var}(S_{j,n})\big)^{k/2}\leq c_{k}+ C B_2^{k/2}
$$
competing the proof of the lemma. 
\end{proof}

\begin{corollary}\label{CorDerPress}
For every $k\geq2$ there exist constants $\ve_k>0$ and $C_k>0$ so that for each $n\in\bbN$ and $t\in[-\ve_k,\ve_k]$, 
$$
|\Pi_{1,n}^{(k)}(it)|\leq C_k\sig_n^2. 
$$
Hence, with $\tilde\Pi_{n}(t)=\Pi_{1,n}(it/\sig_n)$ he have
$$
\sup_{t\in[-\ve_k\sig_n, \ve_k\sig_n]}|\tilde\Pi^{(k)}_n(t)|\leq C_k\sig_n^{-(k-2)}.
$$
\end{corollary}
\begin{proof}
Fix some $k\geq2$.  Let $B_1$ and $B_2$ be large constants so that 
 Lemma \ref{Taylor Lemma2} holds.
 Let $r_0$ be the constant specified in Lemma \ref{Taylor Lemma2}.  
Let $I_1,I_2,...,I_{m_n}$ be disjoint intervals whose union cover $\{1,...,n\}$ so that 
\[
B_1\leq \text{Var}(S_{I_l})\leq B_2
\]
where for each $l$ we set $\DS S_{I_l}=\sum_{j\in I_l}f_j(X_j,X_{j+1})$. 
Note that it is indeed possible to find such intervals
if $B_1$ and $B_2/B_1$ are sufficiently large because  of Theorem \ref{Thm2.1}. Indeed, with $u_n^2$ denoting the structural constants appearing there, there are constants $C_1,C_2>0$ so that for any $n\geq 3$ and $j$,
\begin{equation}\label{Block Var}
 C_1^{-1}\sum_{m=j}^{j+n-1}u_m^2-C_2\leq \text{Var}(S_{j,n})\leq C_1\sum_{m=j}^{j+n-1}u_m^2+C_2.
\end{equation}
It is also clear that $m_n/\sig_n^2$ is uniformly bounded away from $0$ and $\infty$ (if $n$ is large enough). Now, by Lemma \ref{Taylor Lemma2}  there are $\ve_k>0$ and $A_k>0$ so that for each $1\leq l\leq m_n$ and $t\in[-\ve_k,\ve_k]$,
$$
\left|\sum_{j\in I_l}\Pi_{j}^{(k)}(it)\right|\leq A_k.
$$
Hence, 
$\DS 
|\Pi_{1,n}(it)|\leq \sum_{l}\left|\sum_{j\in I_l}\Pi_{j}^{(k)}(it)\right|\leq A_km_n\leq C_k\sig_n^2.
$
\end{proof}

\subsection{Verification of Assumption \ref{GrowAssum}}\label{SecVer}
\begin{proof}[Proof of Proposition \ref{VerifProp}]
Since both sides of \eqref{Exp} with $j=1$ are analytic, $|\del_{1,n}(z)|\leq C|z|\del^n$ for some $\del\in(0,1)$ and $C>0.$ Moreover $\mu_1(h_1^{(0)})=1$. Hence, if $|z|$ is small enough then 
$$
\ln\bbE[e^{zS_n}]=\Pi_{1,n}(z)+G_{n}(z)
$$
where $G_n(z)=\ln\big(\mu_1(h_1^{(z)})+\del_{1,n}(z)\big)$, which is an analytic and uniformly bounded function around the origin (uniformly in $n$).
 Thus  Proposition \ref{VerifProp} follows from Corollary~\ref{CorDerPress}.
\end{proof}

\begin{corollary}\label{EdgeProp1}
Let $r\geq1$.
Suppose that for any $B>0$ and $\del>0$ small enough,
\begin{equation}\label{Cond2}
\int_{\del\leq |x|\leq B\sig_n^{r-1}}|\bbE(e^{ix S_n})/x|dx=o(\sig_n^{-r}).
\end{equation}
Then
\begin{equation}\label{genExp}
\sup_{t}\left|\bbP((S_n-\bbE[S_n])/\sig_n\leq t)-\Phi(t)-\sum_{j=1}^r\sig_n^{-j} P_{j,n}(t)\phi(t)\right|=o(\sig_n^{-r})
\end{equation}
where $\Phi$ and $\phi$ are the standard normal distribution and density function, respectively, and 
%the polynomials 
$P_{j,n}(t)=P_j(t,\hat S_n)$ are the Edgeworth polynomials of 
$\bar S_n=S_n-\bbE[S_n].$
\end{corollary}
Corollary \ref{EdgeProp1} follows from Proposition \ref{EdgeProp} since $S_n$ verifies Assumption \ref{GrowAssum}.

\subsection{A Berry-Esseen theorem and Expansions of order $1$}
\begin{proof}
[Proof of Theorems  \ref{BE} and \ref{Edg1}]

First, Theorem \ref{BE} follows from Propositions \ref{VerifProp} and \ref{BE prop}.

Next, %in order to prove Theorem \ref{Edg1},
applying Theorem 3.5 and inequality (4.2.7) from \cite{DS} we see that if $\{f_n\}$ is irreducible then condition \eqref{Cond2}  with $r=1$
is satisfied. This proves the result.
% for $k=0$. To obtain the result
%for arbitrary $k$ we note that replacing first $k$ terms by $0$ does not change the reducibility of
%$\{f_n\}$ since $S_k$ is tight.
\end{proof}

\section{High order expansions for summands with small essential supremum, proof of Theorem \ref{ThmEssSup} and \ref{EgThm1}}
\label{ScLInf}

\subsection{Existence of expansions.}

 Recall \eqref{d n def}. 
In order to prove Theorem \ref{ThmEssSup},
we need the following:

\begin{lemma}\label{ThLemma}
\cite[eq. (3.3.7)]{DS}
$\exists \delta>0$ s.t.  if $\|f_n\|_\infty |\xi|\leq\delta$ then
$\DS d_n^2(\xi)\geq \frac{\xi^2 u_n^2}{2}.$
\end{lemma}

\begin{proof}[Proof of Theorem \ref{ThmEssSup}]
Let us fix some $r<\frac{1}{1-2\be}$, and take some $r<r_0<\frac{1}{1-2\be}$.
We claim that there are constants $c,C>0$ so that for all $N$ large enough we have
$$ \left|\Phi_N(\xi)\right|\leq \exp\left(-c \xi^2 V_N\right)  \text{ for } |\xi|\leq C\sigma_N^{r_0-1}. $$
This is enough for the Edgeworth expansion of order $r$ to hold by Corollary \ref{EdgeProp1}.

In order to prove the claim,
 let  $ N_0=N_0(N)$ be the smallest positive integer such that
$\DS \sigma_N^{r_0-1} \|f_n\|_\infty\leq \delta$ for all $n>N_0$
where $\delta$ is the number from Lemma \ref{ThLemma}.
Then, since $\|f_n\|=O(n^{-\beta})$ 
\begin{equation}
\label{N0Bound}
 N_0= O\left(\sig_N^{\frac{r_0-1}{\beta}}\right)=
O\left( V_N^{\frac{r_0-1}{2\beta}}\right).
\end{equation}
 Let us show now that $N_0=o(N)$, which in particular yields that $N_0<N/2$ if $N$ is large enough.
The assumption that $\|f_n\|_\infty=O(n^{-\be})$ also implies that $u_n^2=O(n^{-2\be})$ and so by \eqref{U N V N relation},
 \begin{equation}
 \label{VN1-}
 V_N=O(N^{1-2\be}).
\end{equation} 
 Combining this with $r_0<\frac{1}{1-2\be}$ we see that $\sig_N^{\frac{r_0-1}{\beta}}=O(N^\ka)$, 
 where 
 \begin{equation}
\label{KappaES} 
 \kappa=\frac{(r_0-1)}{2\beta} (1-2\beta)=
 1-\frac{1-r_0(1-2\beta)}{2\beta}<1.
\end{equation} 
    Therefore, $N_0=O(N^\ka)$.

Next, let us write
$$
\sum_{n=N_0+1}^{N}u_n^2=\sum_{k=0}^{3}\,\sum_{\substack{N_0<n\leq N,\\n\text{ mod }4=k}}u_n^2:=\sum_{k=0}^{3}U_{N_0,N,k}.
$$
Let $k_N$ be so that $U_{N_0,N,k_N}=\max\{U_{N_0,N,k}:\,0\leq k\leq 3\}$. Then by  \eqref{U N V N relation} there are constants $C,D>0$ so that
\begin{equation}
\label{VarReminder}
V(S_N-S_{N_0})\leq CU_{N_0,N,k_N}+D.
\end{equation}
 Combining \eqref{u n rel}, Lemma \ref{ThLemma}, and \eqref{VarReminder}
we see that the characteristic function of $S_N$ satisfies
\begin{equation}
\label{PhiDeltaV}
 \left|\Phi_N(\xi)\right|\leq \exp\left(-c \xi^2 V(S_N-S_{N_0})\right) \text{ for } |\xi|\leq C\sigma_N^{r_0-1}
\end{equation} 
where $C>0$ is some constant which depends on $\be$, $r_0$, and $\ve_0$ but not on $\xi$ or $N$.
Note that by Lemma \ref{MixLemma} we have
$$ V_N=V_{N_0}+V(S_{N}-S_{N_0})+2\text{Cov}(S_{N_0}, S_{N}-S_{N_0})=
V_{N_0}+V(S_{N}-S_{N_0})+O(1). $$
It follows that
$$ V(S_{N}-S_{N_0})=V_N-V_{N_0}+O(1). $$
On the other hand, by \eqref{VN1-},
$$ V_{N_0} \leq N_0^{1-2\be}
\leq C' V_N^{\kappa} $$
where $\kappa$ is given by \eqref{KappaES}.
 Therefore $\DS V(S_{N}-S_{N_0})=V_N+O\left(V_N^\kappa\right).$
%$\DS \kappa=\frac{(r_0-1)}{2\beta} (1-2\beta)<1$ due to \eqref{BetaNearHalf}.
Combining this with \eqref{PhiDeltaV} gives
$$ \left|\Phi_N(\xi)\right|\leq \exp\left(-c \xi^2 (V_N+O(V_N^\kappa))\right)  \text{ for } |\xi|\leq C\sigma_N^{r_0-1}$$
and the claim follows since $\ka<1$.
\end{proof}

\subsection{Optimality.}
\begin{proof}[Proof of Theorem \ref{EgThm1}]
Fix some $0<\be<1/2$, and take an integer $s>\frac{1}{1-2\be}$. Then 
$$
s_\be:=(s-1)\left(\frac12-\be\right)>\be.
$$
Take 
$c \in(\be, s_\be)$. 
Set $q_n=2^{[c\log_2 n]}$ and $p_n=[n^{-\beta} q_n]$. Let 
$$
a_n=\frac{p_n}{q_n}.
$$
Since $c>\be$ we have 
$$
n^{-\be}(1+o(1))=n^{-\be}-2^{-[c\log_2 n]}\leq a_n\leq n^{-\be}.
$$
Let $Y_n$ be an iid sequence so that $P(Y_n=\pm 1)=\frac12$. Set 
$$
X_n=a_nY_n=\frac{p_n}{q_n}Y_n.
$$
Then, $\bbE[X_n]=0$, $|X_n|=a_n\asymp n^{-\be}$ and $V(X_n)=a_n^2\asymp n^{-2\be}$.
Next, since $q_n$ divides $q_N$ if $n\leq N$ we have
$$
q_N S_N=S_N2^{[c\log_2 N]}\in\bbZ
$$
and so the minimal jump of $S_N$ is at least $\frac{1}{q_N}$. Therefore, if $\alpha_N$ is a possible value of $S_N$ then 
$$
\bbP\big(S_N\in(\al_N,\al_N+\frac12 2^{-[c\log_2 N]}]\big)=0.
$$
On the other hand, if $S_N$ obeyed an expansion of order $s$
then, choosing  $\al_N=O(\sig_N)$ and denoting
  $\eps_N=2^{-[c\log_2 N]}\sig_N^{-1}$, we would get
$$
0=\bbP\big(S_N\in(\al_N,\al_N+\frac12 2^{-[c\log_2 N]}]\big)=\bbP\big(S_N/\sig_N\in (\al_N/\sig_N,\al_N/\sig_N+\eps_N]\big)
$$
$$
=\bbP\big(S_N/\sig_N\leq \al_N/\sig_N+\eps_N\big)-\bbP\big(S_N/\sig_N\leq \al_N/\sigma_N\big)
$$
$$
=\Phi(\al_N/\sig_N+\eps_N)-\Phi(\al_N/\sig_N)$$
$$+\frac1{\sqrt 2\pi}
\sum_{j=1}^s\left(P_{j,N}(\al_N/\sig_N+\eps_N)e^{-\frac12(\al_N/\sig_N+\eps_N)^2}-P_{j,N}(\al_N/\sig_N)e^{-\frac12\al_N^2\sig_N^{-2}}\right)\sig_N^{-j}
$$
$$
+o(\sig_N^{-s})\geq C\eps_N+o(\sig_N^{-s})\geq C'2^{-c\log_2 N}\sig_N^{-1}+o(\sig_N^{-s}).
$$
Since $\sig_N^2$ if of order $\DS \sum_{n=1}^{N}n^{-2\beta}\asymp N^{1-2\be}$ we must have
$$
c>\frac{(s-1)(1-2\beta)}{2}=s_\beta
$$  
which contradicts that $c\in(\be,s_\beta)$. Taking $s=s(\beta)$ to be the smallest integer such that $s>\frac1{1-2\beta}$ we see that the expansions of orders $r>\frac1{1-2\beta}$ do not hold.
\end{proof}

\section{High order expansions for H\"older continuous functions on Riemannian manifolds.} 
\label{ScRM}

\subsection{Distribution of H\"older functions.}
 %We begin Section \ref{ScRM} with the proof 
The following estimate plays an important role in the proof of Theorem \ref{EdgThmHold}.

\begin{lemma}
\label{LmLipInside}
For every Riemanian manifold $\cX$ there is a constant $\fc$ such that for each  real-valued function $\varphi$ on $\cX$ with $\|\varphi\|_{\alpha}\leq 1$
and each $t, \eps$
$$ \nu(\varphi\in [t, t+\eps])\geq  \fc \eps^{1/\al} \min(\nu(\varphi\geq t+\eps), \mu(\varphi\leq t)) $$
where $\nu$ is the normalized
Riemannian volume on $\cX.$
\end{lemma}

\begin{proof}
Since $\cX$ is compact, it can be covered by a finite number of coordinate charts.
Hence for any given $\ve'$ we can cover $\cX$ by the $C^r$ images of coordinate cubes 
of size $\ve'$
%sets whose boundaries are  $C^r$-images of pieces of intersection of  cubes, 
so that the multiplicity of the cover is bounded by a constant $\fK$ which is independent of $\ve'$.

Now, let $\ve'=\delta \eps^{1/\al}$ where $\delta$ is so small that the diameter of each partition element
is smaller than $\eps^{1/\al}/2.$ Consider the cover  of $\cX$ described above and
let $A$ be the union of all cover elements $Q$ such that $\varphi(x)\geq t+\frac{\eps}{2}$
for each $x\in Q$ and $\cS$ be the union of all partition elements which intersect $\partial A.$ 
By the  Isoperimetric Inequality,
$$\Area(\partial A)\geq \frac{\fh}{\fK} \min(\nu(A), \nu(A^c))\geq 
\frac{\fh}{ \fK} \min(\nu(\varphi\geq t+\eps), \nu(\varphi\leq t)) $$
where $\fh$ is the Cheeger constant of $\cX.$ On the other hand, there exists a constant $\ka$ which does not depend on $\ve$ or $\al$ so that
$$ \Area(\partial A)\leq \Area(\partial\cS)\leq {\fK}  \kappa \eps^{1/\al} \nu(\cS) $$
since for each  cover element $Q\subset \cS$ we have
$$ \Area(\partial S\bigcap \partial Q)\leq 
\kappa \eps^{1/\al} \nu(Q). $$
Since $\varphi\in [t, t+\eps]$ on $\cS$ the result follows.
\end{proof}

\subsection{Proof of Theorem \ref{EdgThmHold}.}
For the rest of Section \ref{ScRM} we consider the following setting.
Let $\{X_n\}$ evolve on a compact Riemannian manifold $M$ with transition densities $p_n(x,y)$ bounded and bounded away from $0$. Let us assume that
$f_n: M\times M\to \bbR$ satisfy $\|f_n\|_{\al}:=\max(\sup|f_n|,v_\al(f_n))\leq 1$ for some $0<\al\leq1$.
Denote $\Phi_N(\xi)=\bbE(e^{\xi S_N}).$

\begin{proposition} 
\label{LmEdgeCont}
For all $0<\al\leq 1$ and  $\delta>0$ there exists  $C_1(\al, \del), c_1=c_1(\al,\del)>0$ so that for every $n\in\bbN$ and $\xi\in\bbR$ with $|\xi|\geq \delta$ we have
$$ |\Phi_N(\xi)|\leq  C_1 \exp \left(-c_1V_N|\xi|^{1-\frac{1}\al}\right).$$
\end{proposition}
 Theorem \ref{EdgThmHold} follows by Proposition \ref{LmEdgeCont} together with Corollary \ref{EdgeProp1}.
 \\

The main step in the proof of Proposition \ref{LmEdgeCont} is the following.

\begin{lemma}
\label{LmLipOsc}
For every Riemanian manifold $\cX$ for every $\delta>0$
there is a constant $\hc$ such that for each  real-valued function 
$\varphi$ on $\cX$ with $\|\varphi\|_{\alpha}\leq 1$
and each $\xi$ such that $|\xi|\geq \delta$, 
$$ \iint \sin^2\left(\frac{\xi[\varphi(x_1)-\varphi(x_2)]}{2} \right) \nu(x_1) d\nu(x_2)\leq 
\hc |\xi|^{1-(1/\alpha)} \iint \left[\varphi(x_1)-\varphi(x_2)\right]^2 \nu(x_1) d\nu(x_2). $$
where $\nu$ is the normalized
Riemannian volume on $\cX.$
\end{lemma}

The lemma will be proven in \S \ref{SSHolOsc}, here we complete the proof of the proposition based on the lemma.

Let $\mu$ denote the normalized
Riemannian volume on $M.$
Let us fix some $n\in\bbN$ and consider a random hexagon
$P_n=(x_{n-2},x_{n-1},x_n; y_{n-1},y_n, y_{n+1})$
based at $n.$

Recall \eqref{u n def} and \eqref{d n def}. By uniform ellipticity we have
\begin{equation}
\label{VUNHolder} 
 u_n^2\asymp \int \Gamma^2(P_n) d\mu^6 (P_n) , \quad
d_n^2(\xi) \asymp \int \sin^2 \left(\frac{\xi\Gamma (P_n)}{2}\right) d\mu^6(P_n) . 
\end{equation}
where
$$\Gamma(P_n)=f_{n-2}(x_{n-2},x_{n-1})+f_{n-1}(x_{n-1},x_n)+f_n(x_n,y_{n+1})$$
$$-f_{n-2}(x_{n-2},y_{n-1})-f_{n-1}(y_{n-1},y_n)-f_n(y_n,y_{n+1})$$
is the balance of $P_n.$

Applying Lemma \ref{LmLipOsc} with $\cX=M\times M$ and 
$$\phi_{x_{n-2}, y_{n+1}}(x_{n-1}, x_n)=
f_{n-2}(x_{n-2},x_{n-1})+f_{n-1}(x_{n-1},x_n)+f_n(x_n,y_{n+1})$$
and integrating with respect to $x_{n-2}$ and $y_{n+1}$ 
we obtain $d_n^2(\xi)\geq C \xi^{1-(1/\alpha)} u_n^2 $. 

%Next, we subdivide $\DS U_N^2=\sum_{n=1}^{N}f_n(\xi_n,\xi_{n+1})$ into four sums according to the value of $n\text{ mod }4$.
%Summing over the indexes $n$ of the largest sum of all four, and using   
%and \eqref{U N V N relation} we complete the proof of 
Now  Proposition \ref{LmEdgeCont} follows from \eqref{u n rel} 
\qed

\subsection{The proof of Lemma \ref{LmLipOsc}}
\label{SSHolOsc}
Set
$$ \Delta (x_1, x_2)=|\varphi(x_1)-\varphi_n(x_2)|, \quad \eps=\xi^{-1}, \quad
\fu^2=\iint \Delta^2(x_1, x_2)\nu(x_1) d\nu(x_2), $$
$$ \fd^2(\xi)=\iint \sin^2\left(\frac{\Delta(x_1, x_2)}{2\eps} \right) \nu(x_1) d\nu(x_2).
$$
Decompose $\cX\times \cX=\cA_1\cup \cA_2$  where 
$\DS \cA_1=\left\{(x_1, x_2): \Delta \leq \frac{\eps}{8} \right\}$ and $\cA_2$ is its complement.
We split
the proof of Lemma \ref{LmLipOsc} into two cases.
\\

{\sc Case 1.} If the integral of $\Delta^2$ 
over $\cA_1$ is larger than the integral over $\cA_2$ 
then  using 
%\eqref{VUNHolder} and 
that 
$\left|\frac{\sin t}{t}\right|\geq c$ for $|t|\leq 1/8$ we get  
$$ \fd^2(\xi)\geq  \iint_{\cA_1} \sin^2 \frac{\Delta(x_1, x_2)}{2\eps } d\nu(x_1)d\nu(x_2)\geq 
\frac{c^2 \xi^2}{4} \fu^2 . $$

{\sc Case 2.} Now we assume that the integral over $\cA_2$ is larger. 
Let 
$$ l_k=2^k \eps, \quad
k^*=\text{argmax} \left[ l_k (\nu\times\nu)(\Delta \in [l_k, 2 l_k))\right], \quad \fl=l_{k^*}
$$
and
$$
\fv=\fl (\nu\times\nu)(\Delta \in [\fl, 2 \fl)). $$
Note that under the assumptions of Case 2 we have
\begin{equation}
\label{Var-Vol}
 \fu^2 \leq C_0\sum_{k=-3}^{\log_2 (1/\eps)} \sum_k l_k^2 (\nu\times\nu)(\Delta\in [l_k, 2 l_k))
\leq C_0\fv\sum_{k=-3}^{\log_2 (1/\eps)} 
l_k\leq C \fv.
\end{equation}

Next,
let $\fm$ denote  a median of $\varphi$ 
%\footnote{\color{purple} \sout{Note that by Lemma \ref{LmLipInside} the 
%limits 
%$\DS \fm^-=\sup\left\{t: \nu(\varphi_n\leq t)\leq \frac{1}{2}\right\}$ \\
%and
%$\DS \fm^+=\inf\left\{t: \nu(\varphi_n\leq t)\geq  \frac{1}{2}\right\}$ coincide.}}
\def\fd{{\mathfrak{d}}}with respect to $\nu$,
$\tilde\varphi=\varphi-\fm$ and 
$$\Omega_1=\{\tilde\varphi\leq \fl /2\}, \quad
\Omega_2=\{\tilde\varphi\in (-\fl/2,  \fl/2)\}, \quad  
\Omega_3=\{\tilde\varphi\geq\fl/2 \}. $$
 Let us assume that $\mu(\Omega_3)\geq \mu(\Omega_1),$ the case where the opposite
 inequality holds being similar.
  Since $\Delta (x_1, x_2)< \fl$ for $(x_1, x_2)\in \Omega_2\times \Omega_2$
we have
$$ (\nu\times\nu)(\Delta \geq \fl)\leq 2\left[ \nu(\Omega_1)+\nu(\Omega_3)\right]\leq 4 \nu(\Omega_3). $$ 
Let
$$ \Omega'_j=\{\tilde\varphi\in [(j+0.1)\eps, (j+0.2) \eps]\}
\quad  
\Omega_j''=\{\tilde\varphi\in [(j+0.3)\eps, (j+0.4) \eps]\}
. $$
Since $\fm$ is a median, 
$\DS \nu(\Omega_1\cup\Omega_2) \geq \frac{1}{2}$. Hence
Lemma \ref{LmLipInside} shows that that for $j\leq \frac{\fl}{4\eps}$ we have 
\begin{equation}
\label{OmegaOmega5}
\nu(\Omega_j')\geq c\eps^{1/\al} \nu(\Omega_3), 
\quad
\nu(\Omega_j'')\geq c\eps^{1/\al} \nu(\Omega_3) .
\end{equation}
On the other hand there is a constant $\delta_0>0$ such that for each $x_1\in X$ we have that
$\DS \sin^2\left(\frac{\Delta(x_1, x_2)}{2\eps}\right)\geq \delta_0 $ 
either for all $j$ and all $x_2\in \Omega_j'$ or for all $j$
for all $x_2\in \Omega_j''$.
It follows that if $\cA_2$ dominates then 
$$\fd^2(\xi)\geq \delta_0
\min\left(\sum_{j=1}^{\fl/4\eps}  \nu(\Omega_j'), \sum_{j=1}^{\fl/4\eps} \nu(\Omega_j'')\right)
\geq \hat{c} \fl \mu(\Omega_3)
=\tilde{c}\eps^{1/\al-1}\fl (\nu\times\nu)(\Delta \in [\fl, 2\fl))=\tilde{c}\eps^{1/\al-1} \fv.$$
Combining this with 
\eqref{Var-Vol} we obtain that if $\cA_2$ dominates then
$\DS \fd^2(\xi)\geq c\eps^{1/\al-1} \fu^2.$
\\

Combining the estimates of cases 1 and 2 we obtain the result. \qed

\subsection{Cantor functions.}
\label{SSCantor}
In order to show the optimality of Theorem \ref{EdgThmHold} we need to consider a function $f$
for which the estimate of Lemma \ref{LmLipOsc} is optimal. Moreover, we want $f$  to grow
on a set of small Hausdorff dimension and we want 
the distribution of $f$ to have atoms at  values which
are commensurable with each other. It turns out that Cantor functions studied 
in \cite{Gil, DMRV} satisfy
these conditions. So in this subsection we describe briefly the construction and properties of
Cantor functions.

Let us fix some integers $p\geq 3$, $k\geq1$ and let $q=(p-1)k$.
Set $$\al_{p,p+q}=\frac{1}{\log_p(q+p)}=\frac{\ln p}{\ln (p+q)}.$$ 

On $[0,1]$, let  $C_{p,p+q}$ 
(where $q=(p-1)k$) be the  Cantor set 
of all numbers of the form $\DS x=\sum_{j=1}^\infty\frac{(k+1)a_j}{(p+q)^j}, a_j=0,1,...,p-1$. 
In other words $C_{p,p+q}$  consists of all number in $[0,1]$ which can be written 
in base $p+q$ so that all its digits are divisible by $k+1.$

Let $f$ be the corresponding Cantor function (\cite{Gil}).
 Namely, for $x\in C_{p,p+q}$ we have
$$
f(x)=\sum_{j}\frac{a_j}{p^j},\quad\text{if}\quad  x=\sum_{j}\frac{(k+1)a_{j}}{(p+q)^{j}},
$$
while outside $C_{p,p+q}$ we have 
$$
f(x)=\sup_{y\in C_{p,p+q},\, y\leq x}f(y)=\sum_{j=1}^n\frac{b_j}{p^j}
%$$
\quad \text{where}\quad x=\sum_{j}\frac{x_j}{(p+q)^{j}}, \quad 
b_j=\left[\frac{x_j}{k+1}\right]+1$$ and $n$ is the first index so that $x_n$ is not divisible by $k+1$.
By \cite[Theorem 2]{Gil} (see also \cite{DMRV}), \,$f$ is H\"older continuous with exponent $\al_{p,q}$, which is also the the Hausdorff dimension of $C_{p,q+p}$.
 Note that $f$ is increasing (see \cite[Theorem 1]{Gil}) and that $f(0)=0$ and $f(1)=1$.

\begin{lemma} \label{LmHRCantor}
 For each $n\in \bbN$
\begin{equation}\label{LebPro}
{\mathrm Leb}\{x\in[0,1]: p^nf(x)\not\in\bbZ\}=\left(\frac{p}{p+q}\right)^{n}.
\end{equation}
\end{lemma}

\begin{proof}
To prove the lemma we explain the inductive construction of
 $f$  by following the recursive construction of the set $C_{p,q+p}$.
First, we split $[0,1]$ into $p+q$ closed intervals $I_1,I_2,...,I_{p+q}$ of the same length $\frac1{p+q}$ so that $I_s$ is to the left of $I_{s+1}$ for each $s$. Next, define intervals $J_1,J_2,...,J_{2p+1}$ as follows: 
we define $J_{1}=I_1$, and then inductively $J_{2l+1}=I_{s_{l}+k+1}$, if $J_{2l-1}=I_{s_l}$. For $1\leq l<p$ we define
and $J_{2l}$ to be the union of the intervals $I_s$ between $J_{2l-1}$ and $J_{2l+1}$.
On $J_{2l}$ we define $f|_{J_{2l}}=\frac{l}{p}$. 

The reconstruction of the function $f$ now proceeds by induction. Suppose that at the $n$-th step of the construction $f$ was additionally defined on a union of closed intervals $U_1,...,U_{j_n}$, $j_n=(p-1)p^{n-1}$ of length $k(p+q)^{-n}$ so that $f|_{U_j}=jp^{-n}$,  $U_j$ is to the left of $U_{j+1}$, and the gap between $U_j$ and $U_{j+1}$ is $(p+q)^{-n}$, where $U_0=\{0\}$ and $U_{j_n+1}=\{1\}$. Let us split the interval between $U_j$ and $U_{j+1}$ into equal $p+q$ intervals $I_{1,j,n+1},I_{2,j,n+1},...,I_{p+q,j,n+1}$ of length $(p+q)^{-n-1}$
 so that $I_{s,j,n+1}$ is to the left of $I_{s+1,j,n+1}$ for each $s$. In the $(n+1)$-th step the intervals $J_{1,j,n+1},J_{2,j,n+1},...,J_{2p+1,j,n+1}$ are defined as follows: 
we define $J_{1,j,n+1}=I_{1,j,n+1}$, and then inductively $J_{2l+1,j,n+1}=I_{s_{l}+k+1,j,n+1}$, if $J_{2l-1,j,n+1}=I_{s_l,j,n+1}$. For $1\leq l<p$ we define
and $J_{2l,j,n+1}$ to be the union of the intervals $I_{s,j,n+1}$ between $J_{2l-1,j,n+1}$ and $J_{2l+1,j,n+1}$. On $J_{2l,j,n+1}$ we define  
$$f|_{J_{2l,j,n+1}}=\frac{jp+l}{p^{n+1}}=\frac{j}{p^n}+\frac{l}{p^{n+1}}.$$ 
In view of the above recursive construction of $f$, we obtain \eqref{LebPro} since in the $(n+1)$-th step there are $p^n$  intervals of length $(p+q)^{-n}$ on which $f$ has not been defined yet, and the values of $f$ in all the steps proceeding the $n$-th step do not have the form $s/p^n$ for $s\in\bbZ$.
\end{proof}

\subsection{Optimality.} 
\begin{proof}[Proof of Theorem \ref{EgThm2}]
We first observe that it is enough to prove Theorem \ref{EgThm2} for a dense set of numbers $\alpha$ in $(0,1)$. Indeed, if the theorem holds for $\al$ belonging to a dense set $A$, given $\al_0\in(0,1)$ and 
$r>\frac{\al_0+1}{1-\al_0}$,
we can find $\al\in A$ so that $\al>\al_0$ and 
$r> \frac{\al+1}{1-\al}.$ 
Now, the $\al$-H\"older continuous function we get from  Theorem \ref{EgThm2} with this $\al$ is also $\al_0$-H\"older continuous 
so the result follows.
%since $[-1,1]$ has a finite diameter and $\al>\al_0$. 

Next, let us consider the set 
$$
A=\left\{\frac{\ln p}{\ln(p+q)}:\, p,q\in\bbN, p\geq 3,\, q|(p-1)\right\}.
$$
 This set is dense in $(0,1)$. Indeed, let  $0<a<b<1$. Then, using that $\frac{\ln p}{\ln(q+p)}=\frac{1}{\log_p(q+p)}$, for all $p\geq3$ and  denoting $k=\frac{q}{(p-1)}$, \; $k\in\bbN$ we have 
$$
\frac{1}{\log_p(q+p)}\in(a,b)\,\Longleftrightarrow\,\,
p^{1/b-1}<k+1-\frac{1}{p} <p^{1/a-1}.
$$
 Since $\DS \lim_{p\to \infty} p^{1/a-1}-p^{1/b-1}=\infty$,
we can find  
a number $k$ satisfying the above inequality provided that $p$ is large enough.
%When $p$ is large enough then $p(p^{1/b-1}-1)-p(p^{1/a-1}-1)>2(p-1)$ and so we can always find $k\geq1$ satisfying the above inequality.

Thus we fix some integers $p\geq 3$, $k\geq1$ and let $q=(p-1)k$.
Set $$\al=\al_{p,p+q}=\frac{1}{\log_p(q+p)}=\frac{\ln p}{\ln (p+q)}.$$ 
 Let $f:[-1,1]\to[-1,1]$ be the odd function whose restriction to $[0,1]$ is the Cantor 
function from \S \ref{SSCantor}.
%In what follows we will show that there exists an $\al$-H\"older continuous odd  increasing function 
%$f:[-1,1]\to[-1,1]$ 
%so that $f(-1)=-1$, $f(1)=1$ and $S_n=S_nf$ described in Theorem \ref{EgThm2} does not obey the  Edgeworth expansions of any order $r>\frac{1+\al}{1-\al}$.
%We will first define $f$ on $[0,1]$. After this is done, we will set $f(x)=-f(-x)$ for $x\in[-1,0]$, so the resulting function will be odd.
We will now show that $S_nf$ does not obey Edgworth expansions of any order $r>\frac{\al+1}{\al-1}$.
Let $r=r(\al)$ be the smallest integer so that $r>\frac{\al+1}{\al-1}$, where $\al=\al_{p,q}$. Let us take
$\frac{\al}{1-\al}<c<\frac r2-\frac12$ and set
 $k_N=p^{[c\log_p N]}$. 
Then
%Now,  since $k_n$ divides $k_N$ for all $n\leq N$,
%with $\rho=\frac{p}{p+q}$, we have
$$
\bbP(k_NS_N\not\in\bbZ)\leq N \bbP(k_N f\not\in\bbZ)=N \left(\frac{p}{p+q}\right)^{[c\log_p N]}
=O\left(N^{1-[(1/\alpha)-1]c}\right)=o_{N\to\infty} (1) $$ 
where the second step follows from Lemma \ref{LmHRCantor} and the last step follows
%for some $t_0>0$, where the second inequality hold true
since
$$c\left(\frac{1}{\al}-1\right)=\frac{c(1-\al)}{\al}>1.$$ 
%which implies that $\rho^{[c\log_p n]}=O(n^{-1-\ve})$ for some $\ve>0$.

Let $p_N=p^{[c\log_p N]}\sig_N=k_N\sig_N$ which is of order $N^{c+1/2}$. Then 
$$
 \lim_{N\to\infty} \bbP(S_N/\sig_N\in (p_N)^{-1}\bbZ)=1.
$$
Thus, by considering points in $(p_N)^{-1}\bbZ$ which are of order $1$, we find that if $C$ is large enough then denoting
$$
 m_N=\text{argmax}\{\bbP(S_N/\sig_N=k/p_N): |k/p_N|\leq C\}
$$
 and recalling that $c+\frac{1}{2}>r$
we have 
\begin{equation}\label{InCon}
\bbP(S_N/\sig_N={m_N}/p_N)\geq C_1p_N^{-1}\geq C_2N^{-c-1/2}\geq {C_3}\sig_N^{-r}
\end{equation}
where $C_1,C_2$ and $C_3$  are positive constants. On the other hand, if $S_N$ 
obeyed expansions of order $r$ then
$$
\bbP\left(\frac{S_N}{\sig_N}=\frac{m_N}{p_N}\right)
\leq\lim\sup_{\del\to 0^+}
\left[\bbP\left(\frac{S_N}{\sig_N}\leq \frac{m_N}{p_N}\right)-
\bbP\left(\frac{S_N}{\sig_N}\leq \frac{m_N}{p_N}-\delta \right)\right]
=o(\sig_N^{-r})
$$
which is inconsistent with \eqref{InCon}.
\end{proof}

\end{document}